\documentclass[12pt]{amsart}
\pdfoutput=1

\usepackage[utf8]{inputenc}
\usepackage{indentfirst}
\usepackage{amsmath}
\usepackage{amsthm}
\usepackage{amssymb}
\usepackage{amsfonts}
\usepackage{array}
\usepackage{microtype}
\usepackage{fancyhdr}
\usepackage{tikz-cd}
\usetikzlibrary{arrows}
\usetikzlibrary{positioning}
\usepackage{xcolor}
\usepackage[colorlinks]{hyperref}
\usepackage{mathtools}
\usepackage{thmtools}

\usepackage{import}
\usepackage{xifthen}
\usepackage{pdfpages}
\usepackage{transparent}
\usepackage[
  margin=1in,
  includefoot,
  footskip=30pt,
]{geometry}

\usepackage{csquotes}

\usepackage[style=alphabetic, citestyle=alphabetic]{biblatex}
\hypersetup{
  colorlinks = true,
  citecolor = {purple!70!black},
  linkcolor = {red!60!black}
}

\usepackage{annotate-equations}

\newcommand*\smallcircled[1]{\tikz[baseline=(char.base)]{
    \node[shape=circle,draw,inner sep=2pt] (char) {\tiny #1};}}

\newcommand{\A}{\mathbb{A}}
\newcommand{\C}{\mathbb{C}}
\renewcommand{\k}{\Bbbk}
\renewcommand{\O}{\mathcal{O}}
\renewcommand{\P}{\mathbb{P}}

\newcommand{\Z}{\mathbb{Z}}

\newcommand{\Spec}{\mathrm{Spec}\,}

\newcommand{\RHom}{\mathrm{RHom}}

\newcommand{\intRHom}{\mathrm{R}\mathcal{H}\mathrm{om}}
\newcommand{\RGamma}{\mathrm{R}\Gamma\,}

\newcommand{\Dbcoh}{D^b_{\!\mathrm{coh}}}

\newcommand{\Tot}{\mathrm{Tot}}
\newcommand{\GL}{\mathrm{GL}}

\newcommand{\mA}{\mathcal{A}}

\newcommand{\dual}{{\scriptstyle\vee}}
\newcommand{\iso}{\simeq}
\newcommand{\caniso}{\cong}

\newcommand{\monoarrow}{\hookrightarrow}
\newcommand{\epiarrow}{\twoheadrightarrow}

\newcommand{\rk}{\operatorname{rk}}

\newcommand{\coker}{\operatorname{coker}\,}

\newcommand{\Gr}{\operatorname{Gr}}
\newcommand{\Fl}{\operatorname{Fl}}

\newcommand{\aff}{\mathrm{aff}}

\newcommand{\Sym}{\operatorname{Sym}}
\newcommand{\multl}{\mathrm{mult}_{\Lambda-2\Z}}
\newcommand{\diagsl}{\mathrm{diags}_{\Lambda-2\Z}}

\declaretheoremstyle[
headformat=\NUMBER.\,\NAME\NOTE,
postheadspace=.5em,
spaceabove=6pt,
headfont=\normalfont\small\scshape,
notefont=\normalfont\small\mdseries, notebraces={(}{)},
bodyfont=\normalfont\itshape
]{plainswap}
\declaretheoremstyle[
headformat=\NAME\NOTE,
postheadspace=.5em,
spaceabove=6pt,
headfont=\normalfont\small\scshape,
notefont=\normalfont\small\mdseries, notebraces={(}{)},
bodyfont=\normalfont\itshape
]{nonumplainswap}
\declaretheoremstyle[
headformat=\NUMBER.\,\NAME\NOTE,
postheadspace=.5em,
spaceabove=6pt,
headfont=\normalfont\small\scshape,
notefont=\normalfont\mdseries, notebraces={(}{)},
bodyfont=\normalfont
]{definitionswap}
\declaretheoremstyle[
headformat=\NAME\NOTE,
postheadspace=.5em,
spaceabove=6pt,
headfont=\normalfont\itshape,
notefont=\mdseries, notebraces={(}{)},
bodyfont=\normalfont
]{myremark}

\declaretheorem[style=plainswap, name=Theorem, sharenumber=subsection]{theorem}

\declaretheorem[style=plainswap, numberlike=theorem, name=Proposition]{proposition}

\declaretheorem[style=plainswap, numberlike=theorem, name=Lemma]{lemma}
\declaretheorem[style=plainswap, numberlike=theorem, name=Corollary]{corollary}

\declaretheorem[style=plainswap, numberlike=theorem, name=Conjecture]{conjecture}

\theoremstyle{definition}
\declaretheorem[style=definitionswap, numberlike=theorem, name=Definition]{definition}

\theoremstyle{myremark}
\newtheorem*{remark}{Remark}

\theoremstyle{remark}
\numberwithin{equation}{theorem}

\usepackage{quiver}

\begin{document}

\title{Towards homological projective duality for $\Gr(2, 2n)$}
\author{Dmitrii Pirozhkov}

\begin{abstract}
Consider a Grassmannian~$\mathrm{Gr}(2, V)$ for an even-dimensional vector space~$V$. Its derived category of coherent sheaves has a Lefschetz exceptional collection with respect to the Pl\"ucker embedding. We consider a variety~$X_1$ of pairs consisting of a degenerate~$2$-form on~$V$ and a line in its kernel. Note that~$X_1$ is generically a~$\mathbb{P}^1$-fibration over the Pfaffian variety of degenerate~$2$-forms on~$V$. We construct an exceptional collection of coherent sheaves on~$X_1$ such that the subcategory of~$\Dbcoh(X_1)$ generated by that collection is conjecturally equivalent to the homologically projectively dual category of the Grassmannian.
\end{abstract}

\maketitle

\section{Introduction}
\label{sec:introduction}

\subsection{Homological projective duality}
Homological projective duality is a deep theory that relates certain objects, so-called \emph{Lefschetz categories}, over some projective space~$\P(V)$ with Lefschetz categories over the dual projective space~$\P(V^\dual)$. Ignoring some (very important!) technicalities for the purpose of this introduction we say that a Lefschetz category over~$\P(V)$ is a pair: a triangulated category~$\mA$ with an action of the tensor triangulated category~$(\Dbcoh(\P(V)), \otimes)$, and a Lefschetz semiorthogonal decomposition of~$\mA$, i.e., a semiorthogonal decomposition that is in some specific sense compatible with the action of the twist by~$\O_{\P(V)}(1)$ on~$\mA$. Proper definitions along these lines can be found in~\cite{noncommutativehpd}; this formulation is slightly more abstract than the original one from the foundational paper~\cite{hpd} by Kuznetsov, where the preference was given to Lefschetz categories arising from geometry.

Homological projective duality, below often abbreviated as HPD, defines an explicit construction that takes a (sufficiently nice) Lefschetz category over~$\P(V)$ and outputs a Lefschetz category over the dual projective space~$\P(V^\dual)$. This construction has many nice properties, see~\cite{hpd} or~\cite[Part~II]{noncommutativehpd} for details. Both components of the dual Lefschetz category, i.e., both the $\Dbcoh(\P(V^\dual))$-linear category and the Lefschetz semiorthogonal decomposition of it, depend on the choice of the Lefschetz decomposition in the original category.

While the construction of the duality and proof of its properties are purely abstract results, most applications of the theory of homological projective duality, as well as the motivating examples that led to its creation, are of geometric origin. Given a variety~$X$ with a map~$f\colon X \to \P(V)$ to the projective space, we can consider the category~$\Dbcoh(X)$ as a~$\Dbcoh(\P(V))$-linear category. For any Lefschetz semiorthogonal decomposition of $\Dbcoh(X)$ with respect to this action we construct the homologically projectively dual Lefschetz category over $\P(V^\dual)$. A natural question is to find a geometric interpretation for the dual category.

The original construction is, in a sense, geometric: we take the universal hyperplace section variety~$\mathcal{X} \subset \P(V) \times \P(V^\dual)$ of~$X$ and define the HPD category to be the orthogonal to some specific list of admissible subcategories in~$\Dbcoh(\mathcal{X})$. Still, for some specific~$X$ it is possible to do better. For example, when~$X \subset \P(V)$ is a linear subspace (with a natural Lefschetz decomposition), its HPD category is the derived category of coherent sheaves on the orthogonal linear subspace~$X^\perp \subset \P(V^\dual)$. This description is better at least in the sense that it is more ``economical'': instead of considering a high-dimensional variety over~$\P(V^\dual)$ and taking a quotient of its derived category by some relatively large collection of subcategories, we give a direct answer using a lower-dimensional variety. A search for more direct descriptions of the HPD categories of certain natural varieties, such as Grassmannian varieties, is a major part of the study of homological projective duality. Descriptions like that led to many applications. For an early example, a description of the HPD category of the Grassmannian~$\Gr(2, 7)$ as a non-commutative resolution of singularities of the Pffaffian variety~\cite{kuznetsov2006gr2l} was used to construct a pair of derived equivalent non-birational Calabi--Yau threefolds in~\cite{borisov-caldararu}. More examples can be found in~\cite[Sec.~5]{thomas}.

\subsection{Grassmannian of two-dimensional subspaces}
An interesting family of examples is given by the Grassmannians of two-dimensional subspaces in a vector space. Let $V$ be a vector space of dimension $N$ over an algebraically closed field $\k$ of characteristic zero, and consider the Pl\"ucker embedding $\Gr(2, V) \subset \P(\Lambda^2 V)$. Let us recall the definition of the Lefschetz semiorthogonal decomposition~\cite{hpd}:

\begin{definition}
  Let $T$ be a triangulated category with an endofunctor $\Phi\colon T \to T$. A \emph{Lefschetz semiorthogonal decomposition of $T$ with respect to $\Phi$} is a semiorthogonal decomposition of the form
  \[
    T = \langle \mA_0, \Phi(\mA_1), \Phi^2(\mA_2), \ldots, \Phi^{m-1}(\mA_{m-1}) \rangle,
  \]
  where~$\mA_0 \supset \mA_1 \supset \mA_2 \supset \ldots \supset \mA_{m-1}$ is a chain of subcategories of~$T$. If~$T$ is a~$\Dbcoh(\P^n)$-linear category, a Lefschetz semiorthogonal decomposition is implicitly assumed to be with respect to the autoequivalence $- \otimes \O_{\P^n}(1)$.
\end{definition}

There are several options for a Lefschetz semiorthogonal decomposition of $\Dbcoh(\Gr(2, V))$ with respect to the Pl\"ucker embedding. We take the one from \cite{kuznetsov-exc-grassmannians}:
\begin{equation}
  \label{intro:grassmannian decomposition}
  \Dbcoh(\Gr(2, V)) = \langle \mA_0, \mA_1(1), \ldots, \mA_{N-1}(N-1) \rangle,
\end{equation}
where for odd $N = 2n+1$ all subcategories $\mA_0 = \ldots = \mA_{N-1}$ are generated by the exceptional collection
\[
  \langle \O, Q, \Lambda^2 Q, \ldots, \Lambda^{n-1} Q \rangle,
\]
and for even $N = 2n$ we have
\[
  \mA_0 = \ldots = \mA_{n-1} = \langle \O, Q, \ldots, \Lambda^{n-1} Q \rangle, \qquad \mA_n = \ldots = \mA_{N_1} = \langle \O, Q, \ldots, \Lambda^{n-2} Q \rangle.
\]

It was conjectured in \cite{kuznetsov2006gr2l} that the HPD category to $\Gr(2, V)$ equipped with this Lefschetz semiorthogonal decomposition should be a noncommutative resolution of singularities of the Pfaffian variety $\mathrm{Pf}(V) \subset \P(\Lambda^2 V^\dual)$, i.e., the variety of~$2$-forms on~$V$ whose rank is smaller than~$2 \lfloor \frac{N}2 \rfloor$. In that paper Kuznetsov proved the conjecture for~$N = 6$ and~$N = 7$ by explicitly constructing a Lefschetz exceptional collection in some geometric resolution of singularities of the Pfaffian variety. The proof that the subcategory generated by that collection is homologically projectively dual to the Grassmannian is a complicated computation. Later Rennemo and Segal proved the conjecture for odd values of $N$ in the paper~\cite{rennemo-segal} using a different approach: instead of a ``classical'' resolution of singularities for the Pfaffian variety their starting point is a certain Artin stack birational to the Pfaffian, and they identify the~HPD category of the Grassmannian with a particular subcategory in the matrix factorization category of that stack. They define a candidate subcategory for the case of even-dimensional $V$ as well. Below we describe a different candidate for the HPD category for $\Gr(2, 2n)$.

\subsection{HPD category in the even-dimensional case}
\label{ssec:intro gr2n}
We suggest a conjectural description of the HPD category for the Grassmannian $\Gr(2, V)$ for the case of even $N = 2n$ in terms close to the geometry of the Pfaffian variety. The description is similar in spirit to the exceptional collection that Kuznetsov constructed in~\cite{kuznetsov2006gr2l} for the $\Gr(2, 6)$-case.

Let $X_1$ be the variety of pairs~$\{ ([\omega] \in \P(\Lambda^2 V^\dual), \ell \in \P(V)) \}$, where~$\omega$ is a (degenerate)~$2$-form on~$V$ and~$\ell$ is a one-dimensional subspace in the kernel of~$\omega$. The subscript $1$ in $X_1$ refers to the dimension of $\ell$. Let~$\pi\colon X_1 \to \P(V)$ be the projection morphism. It is easy to see that~$\pi$ is the projectivization of the vector bundle~$\Lambda^2 Q^\dual$, where~$Q$ is the tautological quotient bundle on~$\P(V)$. In particular, $X_1$ is a smooth projective variety. The other projection map $p\colon X_1 \to \P(\Lambda^2 V^\dual)$ factors through the Pffaffian variety of degenerate~$2$-forms on~$V$. Since the dimension of~$V$ is even, the generic degenerate~$2$-form has a two-dimensional kernel, so~$X_1$ is generically a~$\P^1$-fibration over the Pfaffian variety. We define an admissible subcategory~$\mA \subset \Dbcoh(X_1)$ as follows.

Let~$\O(H)$ be the pullback of the line bundle~$\O_{\P(\Lambda^2 V)}(1)$ along the morphism~$p$. For an integer~$0 \leq i \leq n-1$ we define the object~$L_i \in \Dbcoh(X_1)$ as the cone of the morphism
\begin{equation}
  \label{intro:lis}
  \pi^* \Lambda^{2n-i} Q \otimes \O(-(n-i)H) \to \pi^* \Lambda^i Q
\end{equation}
of vector bundles on $X_1$ defined over the point $([\omega], \ell)$ via the pairing with the $(n-i)$'th power of $\omega$. By definition of~$X_1$ the~$2$-form~$\omega$ on~$V$ descends to the~$2$-form on~$\pi^* Q$, so this makes sense. Note that for $i = 0$ the first vector bundle in~\eqref{intro:lis} vanishes and hence~$L_0 \caniso \O_{X_1}$. In fact, we show in Lemma~\ref{lem:li are exceptional} that the morphism~\eqref{intro:lis} is injective and hence each $L_i$ is a coherent sheaf. The main result of this paper is the following theorem.

\begin{theorem}[{ = Theorem~\textup{\ref{thm:lefschetz grouping}}}]
  \label{intro:lefschetz grouping}
  \begin{enumerate}
  \item The sequence $\langle L_0, L_1, \ldots, L_{n-1} \rangle$ is an exceptional collection on $X_1$.
  \item Let $\mA_0$ be the category generated by the sequence above. Then the sequence
    \[
      \langle \mA_0, \mA_0(H), \ldots, \mA_0(k_{\mathrm{max}}H) \rangle
    \]
    is semiorthogonal where $k_{\mathrm{max}} = \binom{2n-1}{2} - 2$.
  \item\ Moreover, that semiorthogonal sequence of subcategories can be extended:
    \begin{equation}
      \langle \mA_0, \mA_0(H), \ldots, \mA_0(k_{\mathrm{max}}H), L_0((k_{\mathrm{max}} + 1)H), \ldots, L_0((k_{\mathrm{max}} + n)H) \rangle
    \end{equation}
    by adding $n$ copies of the twists of $L_0 \caniso \O$. We refer to the subcategory generated by this sequence as $\mA$.
  \end{enumerate}
\end{theorem}

We conjecture that the category $\mA$ is, in fact, the HPD category of $\Gr(2, V)$.

\begin{conjecture}[{ = Conjecture~\textup{\ref{conj:main conjecture}}}]
  \label{intro:main conjecture}
  The subcategory~$\mA \subset \Dbcoh(X_1)$ is linear with respect to the projection morphism~$p\colon X_1 \to \P(\Lambda^2 V^\dual)$, and there exists a~$\Dbcoh(\P(\Lambda^2 V^\dual))$-linear equivalence between~$\mA$ and the homologically projectively dual category to~$\Gr(2, V)$ equipped with the Lefschetz semiorthogonal decomposition~\textup{\eqref{intro:grassmannian decomposition}}.
\end{conjecture}

The semiorthogonalities from Theorem~\ref{intro:lefschetz grouping} are not hard to check numerically, i.e., on the level of $K_0(X_1)$, but lifting this result to the actual semiorthogonality in the derived category is harder. We reduce the semiorthogonality to some algebraic questions related to modules over the symmetric algebra on an exterior square of a vector space.

\subsection{Comparison with related results}
The definition of objects $L_i$ as cones in~\eqref{intro:lis} is inspired by Kuznetsov's description of homological projective duality for $\Gr(2, 6)$ in~\cite{kuznetsov2006gr2l}. To make a more detailed comparison we introduce the variety $X_2$ as the following resolution of singularities of the Pfaffian variety: it is the variety of pairs
\[
  \{([\omega] \in \P(\Lambda^2 V^\dual), K_2 \in \Gr(2, V))\},
\]
where, like in the definition of $X_1$, $\omega$ is a (degenerate) $2$-form on $V$ and $K_2 \subset V$ is a two-dimensional subspace in the kernel of $\omega$.  Here, just like for $X_1$, it is easy to see that the projection morphism $\tilde{\pi}\colon X_2 \to \Gr(2, V)$ is the projectivization of the second exterior power of the tautological quotient bundle $\tilde{Q}$ of rank $2n-2$ on $\Gr(2, V)$. We define $\O_{X_2}(\tilde{H})$ as the relative ample bundle of $\tilde{\pi}$, and we can use the same formula as in~\eqref{intro:lis} to define the objects $\widetilde{L_i} \in \Dbcoh(X_2)$ for $i \in [0; n-1]$ as cones:
\[
  \widetilde{L_i} := [ \tilde{\pi}^* \Lambda^{2n-i} \tilde{Q} \otimes \O_{X_2}(-(n-i)\tilde{H}) \to \tilde{\pi}^* \Lambda^i \tilde{Q} ].
\]
Note that on $X_2$ we have $\widetilde{L_0} \caniso \O_{X_2}$ and $\widetilde{L_1} \caniso \tilde{Q}$.

In \cite{kuznetsov2006gr2l} our variety $X_2$ is called $\tilde{Y}$, and for $n = 3$ the vector bundles $F_0$, $F_1$ and $F_2$ on~$X_2$ defined there in the equation~(13) correspond in our notation to~$\widetilde{L_2}$,~$\widetilde{L_1}$ and~$\widetilde{L_0}$, respectively. The semiorthogonality proved in \cite[Prop.~5.4]{kuznetsov2006gr2l} in this case is an exact analogue of our Theorem~\ref{intro:lefschetz grouping} stated on~$X_2$ instead of~$X_1$. Moreover, \cite[Th.~4.1]{kuznetsov2006gr2l} confirms the~$n = 3$ case of Conjecture~\ref{intro:main conjecture}, at least when stated on~$X_2$ instead of~$X_1$.

This comparison raises the question: why do we work on~$X_1$ instead of~$X_2$? Starting with~$X_2$, a resolution of singularities for the Pfaffian variety, seems more natural for establishing homological projective duality than working with~$X_1$, which is generically a~$\P^1$-fibration over the Pfaffian.  We discuss this question in more detail in subsections~\ref{ssec:x1 decomposition} and~\ref{ssec:comparison with x2} at the end of the paper, but one basic reason is that some computations are easier to perform on~$X_1$. Additionally, since objects $\widetilde{L_i}$ for $i = 0$, $1$ or $n-1$ are somewhat special (e.g., they are locally free sheaves), generalizing the exceptional collection from the $n = 3$ case in~\cite{kuznetsov2006gr2l} to arbitrary $n$ took some experimentation, and the expected semiorthogonal decomposition of $\Dbcoh(X_1)$ mentioned in Subsection~\ref{ssec:x1 decomposition} was helpful in figuring out plausible formulas by checking the semiorthogonality on the level of~$K_0$. In fact, we conjecture (see Subsection~\ref{ssec:comparison with x2}) that the subcategory~$\widetilde{\mA} \subset \Dbcoh(X_2)$ defined analogously to the subcategory~$\mA \subset \Dbcoh(X_1)$ from Theorem~\ref{intro:lefschetz grouping} is in fact equivalent to~$\mA$.

Let us now briefly recall the approach by Rennemo and Segal in \cite{rennemo-segal}. They study homological projective duality for generalized Pfaffian varieties. For the Grassmannian--Pfaffian duality we are after, they fix an auxiliary vector space~$S$ of dimension~$2n-2$ with a symplectic form~$\omega_S$. Any~$2$-form on~$V$ of rank~$2n-2$ is a pullback of~$\omega_S$ along some surjection~$V \epiarrow S$, and the surjection is defined uniquely up to the action of~$\mathrm{Sp}(S, \omega_S)$. The stack~$\tilde{\mathcal{Y}}$ of all, not necessarily surjective, maps from~$V$ to~$S$ modulo the action of the group~$\mathrm{Sp}(S, \omega_S)$, is the key object. Representations of the group~$\mathrm{Sp}(S, \omega_S)$ define vector bundles on~$\tilde{\mathcal{Y}}$, and the candidate for the HPD category of~$\Gr(2, V)$ is defined in terms of \emph{symplectic Schur functors} (see, e.g., \cite[\S{}17.3]{fulton-harris}) applied to~$S$ with respect to a particular set of Young diagrams. For odd-dimensional~$V$ Rennemo and Segal prove that the candidate category is, in fact, the correct answer. They prove a partial result for even-dimensional $V$ in~\cite{rennemo-segal-addendum}

The relation with our definition is not obvious since descending the results concerning the derived category of the stack $\tilde{\mathcal{Y}}$, which is a quite large category, to results about derived categories of simpler geometric objects such as $X_1$ or $X_2$, is, in general, a hard problem. Let us just mention the fact that the objects $L_i$ defined in~\eqref{intro:lis} are similar to the description in~\mbox{\cite[\S{}17.3]{fulton-harris}} of symplectic Schur functors with respect to single-column Young diagrams. It would be interesting to check whether the subcategory $\mA$ defined in Theorem~\ref{intro:lefschetz grouping} is equivalent to the candidate subcategory defined in~\cite{rennemo-segal}.

\subsection{Structure of the paper}
In Section~\ref{sec:bwb} we recall basic properties of Schur functors and state the Borel--Weil--Bott theorem. In Section~\ref{sec:symmetric algebra exterior square} we consider the algebra $\Sym^\bullet(\Lambda^2 V)$ and some maps between graded modules over it from the $\GL(V)$-representation theory point of view. These algebraic observations are necessary in Section~\ref{sec:lefschetz rectangular by computation}, where we prove the main result of this paper, Theorem~\ref{thm:lefschetz grouping}. Finally, in Section~\ref{sec:conjectures} we discuss what else needs to be proved in order to show that our candidate HPD category $\mA$ is the correct one, as well as some other desirable properties of it.

\subsection{Notation and conventions}
Throughout this paper we work over an algebraically closed field $\k$ of characteristic zero. All categories are assumed to be triangulated, all functors are assumed to be derived. For a vector space $V$ we denote by $\P(V)$ the variety of one-dimensional subspaces of $V$. Our convention for Schur functors is that exterior powers correspond to single-column diagrams.

\subsection{Acknowledgments}
I thank Alexander Kuznetsov for suggesting me to work on this problem, many helpful conversations about homological projective duality, and general encouragement. I also thank Anton Fonarev for useful discussions. The work was supported by the Theoretical Physics and Mathematics Advancement Foundation ``BASIS''.

\section{Reminder on Borel--Weil--Bott}
\label{sec:bwb}

In this paper we perform many cohomology computations on Grassmannian varieties. While we mostly work with the projective space $\P(V)$, it is still more convenient to think of it as a Grassmannian $\Gr(1, V)$. This section is a collection of well-known facts about homogeneous bundles, Schur functors, and the Borel--Weil--Bott theorem. A good reference for these facts is the book \cite{weyman}.

\subsection{Schur functors}
Let $\lambda$ be a Young diagram. We denote by $\Sigma^\lambda(-)$ the \emph{Schur functor} corresponding to the partition $\lambda$ (see, e.g., \cite[Ch.~2]{weyman}). Our convention is the following: if $\lambda$ is a single column of length $k$, then $\Sigma^\lambda(V) = \Lambda^k V$ for any vector space $V$, and if $\lambda$ is a single row of length $k$, then $\Sigma^\lambda(V) = \Sym^k V$. While, strictly speaking, Schur functors are endofunctors of the category of vector spaces, the same definition produces endofunctors of the category of vector bundles on some algebraic variety; in fact, Schur functors act on any~$\k$-linear symmetric monoidal category~\cite[(6.5.1)]{sam2012introduction}. The space $\Sigma^\lambda V$ is zero if and only if the number of rows in $\lambda$ is strictly larger than $\dim V$.

It is often convenient to define Schur functors not only with respect to Young diagrams, but with respect to any non-increasing sequence~$\lambda := (\lambda_1, \ldots, \lambda_k)$ of integers (if you prefer, a dominant weight of the group $\GL(k)$). For the definition, let~$\lambda^\prime$ be the sequence given by~$\lambda^\prime_i := \lambda_i - \lambda_k$. Then~$\lambda^\prime$ is a non-increasing sequence of non-negative integers, which we can interpret as a sequence of lengths of rows for a Young diagram. Then we define~$\Sigma^\lambda(V) := \Sigma^{\lambda^\prime}(V) \otimes \det(V)^{\lambda_k}$.

Let $V$ be a vector space over a field $\k$ (which we always assume to be of characteristic zero). Then for any Young diagram $\lambda$ the space $\Sigma^\lambda V$ is an irreducible representation of the group $\GL(V)$, and, after a twist by a sufficiently large power of~$\det(V)$, any (rational) representation splits into a direct sum~$\oplus_{\lambda} (\Sigma^\lambda V)^{\oplus m_\lambda}$ for some multiplicities~$m_\lambda$ (see, for example, theorems~(2.2.9) and~(2.2.10) in \cite{weyman}). In particular, any tensor product~$\Sigma^\lambda(V) \otimes \Sigma^\mu(V)$ and any composition~$\Sigma^\lambda(\Sigma^\mu(V))$ can be expressed as a direct sum of Schur functors of~$V$. The description of the corresponding multiset of Young diagrams for a tensor product is known as Littlewood--Richardson rule, but we, fortunately, have no need of it beyond the following simpler special case:

\begin{definition}
  \label{def:vertical strip}
  Let $\lambda$ be a Young diagram. For a subdiagram $\mu \subset \lambda$ we say that $\lambda/\mu \in \mathrm{VS}_k$ (or, in words, \emph{is a vertical strip of length $k$}) if the complement to $\mu$ in $\lambda$ consists of $k$ boxes, all in different rows.
\end{definition}

\begin{theorem}[Pieri's formula]
  \label{thm:pieri}
  Let $V$ be a vector space. Let $\mu$ be a Young diagram, and let $k \geq 0$ be an integer. There is an isomorphism
  \[
    \Lambda^k V \otimes \Sigma^\mu V \caniso \bigoplus_{\lambda} \Sigma^\lambda V
  \]
  of~$\GL(V)$-representations, where the right-hand side is a multiplicity-free sum taken over all Young diagrams~$\lambda$ containing~$\mu$ such that~$\lambda/\mu$ is a vertical strip of length~$k$.
\end{theorem}
\begin{proof}
  See, for example, \cite[(2.3.5)]{weyman}.
\end{proof}

\begin{theorem}
  \label{thm:plethysm description}
  Let $V$ be a vector space. Let $m \geq 0$ be an integer. Then there is an isomorphism
  \[
    \Sym^m(\Lambda^2 V) \caniso \bigoplus_{\lambda} \Sigma^\lambda V
  \]
  of $\GL(V)$-representations, where the right-hand side is a multiplicity-free sum taken over all Young diagrams~$\lambda$ with $2m$ boxes such that every column of $\lambda$ has even length.
\end{theorem}
\begin{proof}
  See, for example, \cite[(2.3.8)]{weyman}.
\end{proof}

\subsection{Borel--Weil--Bott}
Let $V$ be a vector space of dimension $N$ and consider the Grassmannian $\Gr(k, V)$. Denote by $U \subset \O_{\Gr(k, V)} \otimes V$ the tautological subbundle, and let $Q$ be the tautological quotient bundle.
It is a well-known fact that any irreducible $\GL(V)$-equivariant vector bundle on~$\Gr(k, V)$ is isomorphic to the tensor product~$\Sigma^\lambda U^\dual \otimes \Sigma^\mu Q^\dual$, where~$\lambda$ is a non-increasing sequence of~$k$ integers and~$\mu$ is a non-increasing sequence of~$N-k$ integers. For example, the bundle $U$ corresponds to a pair $(\lambda, \mu)$, where $\lambda = (0, \ldots, 0, -1)$ and $\mu = 0$. The Borel--Weil--Bott theorem is a fundamental result that describes the cohomology of vector bundles of this form.

\begin{theorem}
  \label{thm:bwb}
  Let $\lambda$ and $\mu$ be non-increasing sequences of integers, of length $k$ and $N-k$, respectively. Let $\alpha$ be the concatenation of $\lambda$ and $\mu$, a sequence of $N$ integers. Denote by $\rho$ the sequence $(N, N-1, \ldots, 1)$ and consider the termwise sum $\alpha + \rho$.
  \begin{itemize}
  \item If the sequence $\alpha + \rho$ is not composed of $N$ distinct integers, then
    \[
      \RGamma(\Sigma^\lambda U^\dual \otimes \Sigma^\mu Q^\dual) = 0.
    \]
  \item If the sequence $\alpha + \rho$ is composed of $N$ distinct integers, let $\sigma \in S_N$ be the unique permutation such that $\sigma(\alpha + \rho)$ is a decreasing sequence. Then
    \[
      H^{i}(\Gr(k, V), \Sigma^\lambda U^\dual \otimes \Sigma^\mu Q^\dual) =
      \begin{cases}
        \Sigma^{\sigma(\alpha + \rho) - \rho} V^\dual, & \text{if $i = l(\sigma)$} \\
        0, & \text{otherwise.}
      \end{cases}
    \]
  \end{itemize}
\end{theorem}

Note that the sequence $\alpha + \rho$ contains some integer twice if and only there exists a pair of indices $i$, $j$ such that $\alpha_i = \alpha_j + i - j$, i.e., the following fragments of $\alpha$ imply the acyclicity:
\[
  \begin{aligned}
    (\ldots \, , x, x + 1, \,\, \ldots) & \rightsquigarrow & \text{acyclic} \\
    (\ldots \, , x, *, x + 2, \,\, \ldots) & \rightsquigarrow & \text{acyclic} \\
    (\ldots \, , x, *, *, x + 3, \,\, \ldots) & \rightsquigarrow & \text{acyclic} \\
    & \text{\ldots} &
  \end{aligned}
\]

\begin{proof}
  See, for example, \cite[Cor.~(4.1.9)]{weyman}.
\end{proof}

In this paper we mostly work with $\Gr(1, V) = \P(V)$, for which~$U$ is the line bundle~$\O_{\P(V)}(-1)$. For convenience we record some sufficient conditions for acyclicity of vector bundles in this case.

\begin{corollary}
  \label{cor:bwb acyclicity}
  Let $V$ be a vector space of dimension $N$ and let~$Q$ be the tautological quotient bundle on~$\P(V)$. Let~$\lambda$ be a Young diagram and let~$k > 0$ be a positive integer.
  \begin{itemize}
  \item If $\lambda$ has strictly less than $k$ rows, then $\O_{\P(V)}(-k) \otimes \Sigma^\lambda Q^\dual$ is an acyclic bundle.
  \item If $\lambda$ has strictly less than $N - k$ rows, then $\O_{\P(V)}(-k) \otimes \Sigma^\lambda Q$ is an acyclic bundle.
  \end{itemize}
\end{corollary}
\begin{proof}
  The first claim follows directly from Theorem~\ref{thm:bwb}. For the second claim note that~$\Sigma^{\lambda} Q \caniso \Sigma^{-\lambda} Q^\dual$, where~$-\lambda$ is a sequence of integers defined by~$(-\lambda)_i := -(\lambda_{N - i})$.
\end{proof}

\section{Symmetric algebra on an exterior square}
\label{sec:symmetric algebra exterior square}

Since the main results and conjectures in this paper are closely related to Pffafian varieties, i.e., varieties of degenerate $2$-forms on some vector space, it is not surprising that the second exterior powers of vector spaces (or vector bundles) play the key role. Many computations that we perform later are reduced to some algebraic or combinatorial properties of the symmetric algebra $\Sym^\bullet(\Lambda^2 V)$ for a vector space $V$ and its relation with the exterior algebra~$\Lambda^\bullet V$. In this section we record some facts about these algebras.
% I would not be surprised if some of these results were known, but I wasn't able to find them in the literature.

\subsection{Young diagrams and multiplicities}
First, we discuss the decomposition of the tensor product $\Lambda^\bullet V \otimes \Sym^\bullet(\Lambda^2 V)$ into a direct sum of irreducible $\GL(V)$-representations. We start with a couple of definitions.

\begin{definition}
  \label{def:mult function}
  Let $\lambda$ be a Young diagram and let $k \geq 0$ be an integer. We define the set~$\diagsl(\lambda, k)$ of subdiagrams of $\lambda$ as follows:
  \[
    \diagsl(\lambda, k) := \{ \mu \subset \lambda \, | \, \mu \, \text{has even columns}, \lambda/\mu \in \mathrm{VS}_k \}.
  \]
  We denote the cardinality of this set by $\multl(\lambda, k) := \#(\diagsl(\lambda, k))$.
\end{definition}

\begin{lemma}
  \label{lem:mult is multiplicity}
  Let $\lambda$ be a Young diagram and let $k \geq 0$ be an integer. Consider the decomposition of the tensor product $\Lambda^k V \otimes \Sym^\bullet(\Lambda^2 V)$ into irrreducible $\GL(V)$-representations. Then the multiplicity of the irreducible representation~$\Sigma^\lambda V$ in this tensor product is equal to~$\multl(\lambda, k)$.
\end{lemma}
\begin{proof}
  The $m$'th symmetric power $\Sym^m(\Lambda^2V)$ as a $\GL(V)$-representation splits into a multiplicity-free direct sum of representations $\Sigma^\mu V$ for all Young diagrams $\mu$ of size $2m$ such that all columns of $\mu$ have even length (Theorem~\ref{thm:plethysm description}). Then by the Pieri's formula (Theorem~\ref{thm:pieri}) in the decomposition of $\Lambda^k V \otimes \Sym^m(\Lambda^2 V)$ into irreducible representations the multiplicity of a diagram~$\lambda$ of size $k + 2m$ is equal to the number of subdiagrams~$\mu \subset \lambda$ such that all columns of~$\mu$ have even length and the complement is a vertical strip of length~$k$, i.e., to the number~$\multl(\lambda, k)$.
\end{proof}

The function $\multl(\lambda, -)$ for a fixed diagram $\lambda$ has some nice properties which are not completely immediate from the definition. We prove them in a series of lemmas.

\begin{lemma}
  If for some odd number $m$ the diagram $\lambda$ has more than one column of height~$m$, then~$\multl(\lambda, -) = 0$ for any argument.
\end{lemma}
\begin{proof}
  In this case removing a vertical strip from $\lambda$ will not change the length of leftmost column of height $m$, hence we will never get a diagram with even columns as the result.
\end{proof}

In this section we will say that a Young diagram $\lambda$ is \emph{mildly odd} if each odd-length column appears in $\lambda$ at most once.

\begin{definition}
  Let $\lambda$ be a Young diagram. Its \emph{height} $\mathrm{ht}(\lambda)$ is the number of non-empty rows in $\lambda$. We define the \emph{profile} of $\lambda$ to be a sequence of positive integers $h(\lambda) = (h_0, \ldots, h_n)$ as follows. Mark the rightmost box in each row of $\lambda$; they form a collection of vertical segments. Let $h_i$ be the length of the $i$'th vertical segment, counting from right to left. Figure~\ref{fig:profile of a diagram} is probably easier to understand than the description in words.
\end{definition}

\begin{figure}
  \label{fig:profile of a diagram}
  \includegraphics{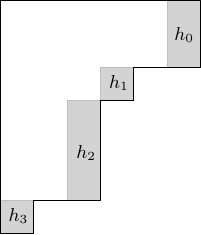}
  \caption{A Young diagram with profile $(2, 1, 3, 1)$}
\end{figure}

\begin{lemma}
  Let $\lambda, \lambda^\prime$ be two mildly odd Young diagrams with the same profile. Then there is a natural bijection between sets
  \[
    \mathrm{diags}_{\Lambda-2\Z}(\lambda, -) \caniso \mathrm{diags}_{\Lambda-2\Z}(\lambda^\prime, -)
  \]
  and hence an equality of numbers
  \[
    \mathrm{mult}_{\Lambda-2\Z}(\lambda, -) = \mathrm{mult}_{\Lambda-2\Z}(\lambda^\prime, -).
  \]
\end{lemma}
\begin{proof}
  This is clear by definition: removing a vertical strip from $\lambda$ depends only on the profile.
\end{proof}

Imagine a diagram $\lambda$ whose rightmost column has odd length. Then clearly if we want to remove something from $\lambda$ so that the remaining diagram would have even columns, we have to remove at least one box from the rightmost column. The following lemma builds upon that observation.

\begin{lemma}
  \label{lem:evenization step}
  Let $h(\lambda) = (h_0, \ldots, h_n)$ be a profile of a mildly odd Young diagram. Let $i$ be the smallest integer such that $h_i$ is odd. Let $h^\prime$ be the following sequence of numbers:
  \[
    \begin{split}
      h_0, \ldots, h_{i-1}, h_i - 1, h_{i+1} - 1, h_{i+2}, \ldots, h_n \quad & \text{if $i \neq n$,} \\
      h_0, \ldots, h_{n-1}, h_n - 1 \quad & \text{if $i = n$.}
    \end{split}
  \]
  Then there is an equality
  \[
    \mathrm{mult}_{\Lambda-2\Z}(h, k) = \mathrm{mult}_{\Lambda-2\Z}(h^\prime, k-1).
  \]
\end{lemma}
\begin{proof}
  For the purpose of this proof it is convenient to think of $\diagsl(\lambda, -)$ as the set of possible vertical strips in $\lambda$ for which the complement has even columns. Let us temporarily call such vertical strips \emph{admissible}. Since $i$ is the smallest integer such that $h_i$ is odd and $\lambda$ is mildly odd, there is a unique column in $\lambda$ of odd length $h_0 + \ldots + h_i$. By definition any admissible vertical strip must contain at least the bottom box of that column.

  Now consider the next largest column of $\lambda$. It has $h_0 + \ldots + h_i + h_{i+1}$ boxes. A vertical strip in $\lambda$ may include up to $h_{i+1}$ boxes from the bottom of this column. However, an admissible vertical strip cannot contain all available $h_{i+1}$ boxes, otherwise the remaining boxes form a column of odd length $h_0 + \ldots + h_i$. This means that an admissible vertical strip does not include the topmost available box in that column, i.e., the box at the level $h_0 + \ldots + h_i + 1$ of that column.

  Thus we have found a row whose rightmost box must be included in any admissible vertical strip, while in the next row the rightmost box must not be included in any admissible strip. This establishes a bijection between admissible vertical strips of size $k$ in $\lambda$ and admissible vertical strips of size $k-1$ in the diagram one gets if we take $\lambda$, find the last box in the column of length $h_0 + \ldots + h_i$, then remove the full row containing it \emph{and} the next row as well; the profile of the resulting diagram is given by the formulas in the statement.
\end{proof}

By construction if $h_i$ is the first odd number in the profile, then any odd numbers in $h^\prime$ have indices strictly bigger than $i$. Hence by repeating the procedure enough times we arrive at a profile that contains only even numbers. We denote the result by $h^{\mathrm{ev}}(\lambda)$. An evident corollary of the lemma above is

\begin{lemma}
  \label{lem:evenization of profile}
  Let $\#\mathrm{oddcolumns}(\lambda)$ be the number of odd-length columns in a mildly odd Young diagram~$\lambda$. Then
  \[
    \mathrm{mult}_{\Lambda-2\Z}(h(\lambda), k) = \mathrm{mult}_{\Lambda-2\Z}(h^{\mathrm{ev}}(\lambda), k - \#\mathrm{oddcolumns}(\lambda)).
  \]
\end{lemma}

Let us record some properties of the $\mathrm{mult}_{\Lambda-2\Z}$ function applied to an even profile.

\begin{lemma}
  \label{lem:even profile properties}
  Let $h = (h_0, \ldots, h_n)$ be an even profile. Let $m\colon \Z_{\geq 0} \to \Z_{\geq 0}$ be the function defined by $m(k) = \mathrm{mult}_{\Lambda-2\Z}(h, k)$. Then:
  \begin{enumerate}
  \item $m$ vanishes on any odd argument;
  \item $m$ is symmetric around the argument $(\sum h_i)/2$;
  \item $m$ is non-decreasing on even integers between $0$ and $(\sum h_i)/2$;
  \item If for some $j$ we have $h_j > \sum_{i \neq j} h_i$, then the function $m$ (restricted to even integers) is constant on the segment~$[ \sum_{i \neq j} h_i, h_j ]$.
  \end{enumerate}
\end{lemma}
\begin{proof}
  It is easy to see from the definition of the $\mathrm{mult}_{\Lambda-2\Z}$ function that for an even profile we have a simple formula:
  \[
    m(k) = \# \{ (y_0, \ldots, y_n) \in \Z^{n+1} \,\, | \,\, y_i \, \text{is even}, 0 \leq y_i \leq h_i, \sum y_i = k \}.
  \]
  Now it's clear that the map $(y_i) \mapsto (h_i - y_i)$ is a bijection between $m(k)$ and $m(\sum h_i - k)$. By dividing everything by $2$ we get that $m(2k)$ is equal to the number of integral points in the hyperplane slice $\sum \tilde{y_i} = k$ of the box
  \[
    [0; h_0/2] \times [0; h_1/2] \times \ldots \times [0; h_n/2].
  \]
  In other words, $m(2k)$ is a convolution of indicator functions of the segments in the formula above. The last two points of the lemma are easy to prove by induction on the number of segments.
\end{proof}

\begin{corollary}
  \label{cor:odd profiles properties}
  Let $\lambda$ be a mildly odd diagram. Then the function $\mathrm{mult}_{\Lambda-2\Z}(\lambda, -)$ has the following properties:
  \begin{enumerate}
  \item It is zero on the segment from $0$ to $\#\mathrm{oddcolumns}(\lambda)$.
  \item It is symmetric with respect to the argument $\left\lceil \mathrm{ht}(\lambda)/2\right\rceil$.
  \item When restricted to the numbers of the same parity as $\#\mathrm{oddcolumns}(\lambda)$, it is nondecreasing on the segment from $0$ to the symmetry point.
  \end{enumerate}
\end{corollary}
\begin{proof}
  Easily follows from Lemmas~\ref{lem:evenization of profile} and~\ref{lem:even profile properties}. The rounding up in point (2) appears due to the special formula for the $i = n$ case in Lemma~\ref{lem:evenization step}.
\end{proof}

\begin{remark}
  One can also state an analogue of the point (4) of Lemma~\ref{lem:even profile properties}: if $\mathrm{ht}(\lambda)$ is an even number, the statement is the same, otherwise one needs to adjust the constants by one depending on whether the number $j$ is the last one or not.
\end{remark}

\subsection{Tautological morphisms}

The following result and its corollaries are indispensable for future computations in Section~\ref{sec:lefschetz rectangular by computation}.

\begin{proposition}
  \label{prop:phi maps}
  Let $0 \leq i \leq n$ be two integers, and let $m \geq 0$ be another one. Consider the following natural transformation of endofunctors of the category of vector spaces:
  \[
    \varphi_{n, i, m}\colon
    \Lambda^{2n - i}(-) \otimes \Sym^m(\Lambda^2(-))
    \to
    \Lambda^i(-) \otimes \Sym^{m + n - i}(\Lambda^2(-))
  \]
  defined using the comultiplication in the exterior algebra
  \[
    \Lambda^{2n-i}(-) \to \Lambda^i(-) \otimes (\Lambda^2(-))^{\otimes (n - i)}
  \]
  and the multiplication
  \[
    (\Lambda^2(-))^{\otimes (n - i)} \otimes \Sym^m(\Lambda^2(-)) \to \Sym^{m + n - i}(\Lambda^2(-))
  \]
  in the symmetric algebra. Then there exists a functorial splitting
  \[
    \ker(\varphi_{n, i, m}) \iso \bigoplus_{\lambda \in K_{n, i, m}} \Sigma^\lambda(-)
  \]
  into a direct sum of Schur functors, where $K_{n, i, m}$ is a multiset of Young diagrams; moreover, all diagrams in $K_{n, i, m}$ have strictly more than $2n$ rows. 
\end{proposition}
\begin{proof}
  The existence of a functorial splitting into a direct sum of Schur functors is a general fact: the morphism $\varphi_{n, i, m}$ is a natural transformation between two \emph{polynomial functors} of degree~$2m + 2n - i$ (see, e.g., \cite[Ch.~I,~App.~A]{macdonald}). By \cite[I.A~(5.3)]{macdonald} the category of polynomial functors of degree $2m + 2n - i$ is equivalent to the category of representations of the symmetric group $S_{2m + 2n - i}$, so $\varphi_{n, i, m}$ corresponds to a morphism between two representations. The decomposition of the kernel of that morphism into irreducible representations of $S_{2m + 2n - i}$ provides the desired functorial splitting of $\ker(\varphi_{n, i, m})$ into a direct sum of Schur functors.

  It remains to show that all Young diagrams in $K_{n, i, m}$ have strictly more than $2n$ rows. Note that if $V$ is a $2n$-dimensional vector space, then $\Sigma^\lambda(V)$ is a zero vector space if and only if $\lambda$ has more than $2n$ rows. Since the splitting of $\ker(\varphi_{n, i, m})$ is functorial, we can prove the claim about diagrams in $K_{n, i, m}$ by showing that the map of vector spaces we get by evaluating the natural transformation $\varphi_{n, i, m}$ on a $2n$-dimensional vector space $V$ has zero-dimensional kernel. In other words, it remains to prove that the morphism
  \[
    \varphi_{n, i, m}(V)\colon
    \Lambda^{2n - i}V \otimes \Sym^m(\Lambda^2V)
    \to
    \Lambda^iV \otimes \Sym^{m + n - i}(\Lambda^2V)
  \]
  is injective if $\dim V = 2n$.

  We prove the injectivity of $\varphi_{n, i, m}(V)$ by a geometric argument. Let $\A := \Spec \Sym^\bullet(\Lambda^2 V)$ be the vector space $\Lambda^2 V^\dual$ considered as an algebraic variety. A trivial vector bundle with fiber~$\Lambda^2 V^\dual$ over~$\A$ has the tautological global section~$s \in H^0(\A, \Lambda^2 V^\dual \otimes \O_A)$. Consider the map of trivial vector bundles
  \[
    \widetilde{\varphi_{n, i}}\colon \Lambda^{2n-i}V \otimes \O_{\A} \xrightarrow{s^{n-i}} \Lambda^iV \otimes \O_{\A}
  \]
  over that space. Over a general point of $\A$, which corresponds to a nondegenerate $2$-form on~$V$, one checks by direct computation that the map of fibers is injective (in fact, an isomorphism). Thus the kernel of the map~$\widetilde{\varphi_{n, i}}$ is a torsion subsheaf of the vector bundle~$\Lambda^{2n-i}V \otimes \O_A$, which means that the kernel vanishes and~$\widetilde{\varphi_{n, i}}$ is injective as a map of coherent sheaves. This implies that the induced map~$H^0(\widetilde{\varphi_{n, i}})$ on global sections is injective. Since the morphism~$\varphi_{n, i, m}(V)$ is a direct summand of~$H^0(\widetilde{\varphi_{n, i}})$, it is also injective.
\end{proof}

For convenience we record the cases where substituting a vector space into $\varphi_{n, i, m}$ results in an injective morphism.

\begin{corollary}
  \label{cor:injective phi maps}
  Let $0 \leq i \leq n$ be two integers, and let $m \geq 0$ be another. Let $V$ be a vector space with $\dim(V) \leq 2n$. Then the morphism
  \[
    \varphi_{n, i, m}(V)\colon
    \Lambda^{2n - i}V \otimes \Sym^m(\Lambda^2V)
    \to
    \Lambda^iV \otimes \Sym^{m + n - i}(\Lambda^2V)
  \]
  is injective.
\end{corollary}
\begin{proof}
  If $\lambda$ is a Young diagram with strictly more than $2n$ rows, then $\Sigma^\lambda(V) = 0$. Thus the statement follows from the universal description of $\ker(\varphi_{n, i, m})$ in Proposition~\ref{prop:phi maps}.
\end{proof}

\begin{remark}
  Proposition~\ref{prop:phi maps} and Corollary~\ref{cor:injective phi maps} are stated in terms of the category of vector spaces, but one easily sees that their analogues also hold in the category of vector bundles on some algebraic variety: this can be proved in a down-to-earth fashion, by defining the maps fiberwise and checking that they glue, or instead we can invoke some abstract nonsense about the category of polynomial functors acting on any symmetric monoidal $k$-linear category (see, e.g., \cite[(6.5.1)]{sam2012introduction} or \cite{baez2023schur}).
\end{remark}

\begin{corollary}
  \label{cor:cone of tautological morphism}
  Let $0 \leq i \leq n$ be two integers, and let $m \geq 0$ be another. Let $V$ be a vector space or a vector bundle on some algebraic variety. The cone of the morphism
  \[
    \varphi_{n, i, m}(V)\colon
    \Lambda^{2n - i}V \otimes \Sym^m(\Lambda^2V)
    \to
    \Lambda^iV \otimes \Sym^{m + n - i}(\Lambda^2V)
  \]
  is isomorphic to the direct sum $\ker(\varphi_{n, i, m}(V))[1] \oplus \coker(\varphi_{n, i, m}(V))$. Both the kernel and the cokernel of $\varphi_{n, i, m}(V)$ can be decomposed into direct sums of Schur functors applied to $V$. If $\lambda$ is a Young diagram with $2n$ or $2n-1$ rows, then it cannot appear in the decomposition of the cone of $\varphi_{n, i, m}$ with nonzero multiplicity.
\end{corollary}
\begin{proof}
  As observed in the proof of Proposition~\ref{prop:phi maps}, all claims from the statement but the last one follow from the fact that the morphism $\varphi_{n, i, m}$ is defined universally, in the category of polynomial functors, which is semisimple.

  Let $\lambda$ be a Young diagram. By Lemma~\ref{lem:mult is multiplicity} the multiplicity of the representation $\Sigma^\lambda V$ in the domain and the codomain of $\varphi_{n, i, m}(V)$ is equal to $\multl(\lambda, 2n-i)$ and $\multl(\lambda, i)$, respectively. By Corollary~\ref{cor:odd profiles properties} the function $\multl(\lambda, -)$ is symmetric with respect to the point $\lceil \mathrm{ht}(\lambda)/2 \rceil$, where $\mathrm{ht}(\lambda)$ is the number of rows in $\lambda$.
  Thus we see that if the diagram $\lambda$ has $2n$ or $2n-1$ rows, it appears with the same multiplicity in the domain and the codomain of~$\varphi_{n, i, m}(V)$. This implies that it appears with the same multiplicity in the kernel and the cokernel of~$\varphi_{n, i, m}(V)$. However, we already know by Proposition~\ref{prop:phi maps} that the multiplicity of any~$\lambda$ with~$2n$ or fewer rows in~$\ker(\varphi_{n, i, m}(V))$ is zero. Thus the multiplicity of~$\lambda$ in the cokernel is also zero.
\end{proof}

\begin{remark}
  Let $\lambda$ be a Young diagram and let $V$ be some vector space. We saw in the proof of Corollary~\ref{cor:cone of tautological morphism} that the multiplicity of the representation $\Sigma^\lambda V$ in the domain and codomain of $\varphi_{n, i, m}(V)$ is equal to $\multl(\lambda, 2n-i)$ and $\multl(\lambda, i)$, respectively. If $\lambda$ is not mildly odd, both those numbers are zero, so assume that $\lambda$ is a mildly odd diagram and let $\mathrm{ht}(\lambda)$ be the number of rows in $\lambda$. Then Corollary~\ref{cor:odd profiles properties} describes the behavior of the function $\multl(\lambda, -)$: it is non-decreasing from $0$ up to the symmetry point $\lceil \mathrm{ht}(\lambda)/2 \rceil$ and then it is non-increasing symmetrically. Thus we see that
  \[
    \begin{aligned}
      \mathrm{ht}(\lambda) \leq 2n-2 & \implies & \multl(\lambda, 2n-i) \leq \multl(\lambda, i) \\
      \mathrm{ht}(\lambda) = 2n - 1 \,\, \text{or} \,\, 2n & \implies & \multl(\lambda, 2n - i) = \multl(\lambda, i) \\
      \mathrm{ht}(\lambda) > 2n & \implies & \multl(\lambda, 2n-i) \geq \multl(\lambda, i).
    \end{aligned}
  \]
  Maybe the morphism $\varphi_{n, i, m}$ is maximally cancellative in the sense that a Young diagram cannot appear in its kernel and its cokernel at the same time; if this were the case, not only would we conclude that its kernel is composed of Schur functors with respect to Young diagrams with strictly more than $2n$ rows, but also that its cokernel is composed of Schur functors with respect to Young diagrams with $\leq 2n-2$ rows. Sadly, I don't know how to prove this.
  % Roma Mihailov has the notion of "carnivorous spectral sequences" which are kind of this exact phenomenon: a morphism in a spectral sequence is replaced by a natural transformation of functors and separately one expects to prove that this natural transformation is maximally cancellative by checking in some exotic cases.
\end{remark}

While Proposition~\ref{prop:phi maps} is concerned only with the kernel of the map $\varphi_{n, i, m}$, we will use some properties of the cokernel as well, at least for small values of $m$.

\begin{lemma}
  \label{lem:coker for small twists}
  Let $0 \leq i \leq n$ be two integers, and let $m \geq 0$ be another one. Let $V$ be a vector space with $\dim(V) \leq 2n$. Consider the morphism
  \[
    \varphi_{n, i, m}(V)\colon
    \Lambda^{2n - i}V \otimes \Sym^m(\Lambda^2V)
    \to
    \Lambda^iV \otimes \Sym^{m + n - i}(\Lambda^2V)
  \]
  defined as in Proposition~\ref{prop:phi maps}. Then the cokernel of $\varphi_{n, i, m}$ can be decomposed into a direct sum~$\oplus \Sigma^\lambda V$ of some Schur functors applied to~$V$, where all diagrams $\lambda$ appearing with non-zero multiplicity have strictly less than $2n - i + m$ rows.
\end{lemma}

\begin{remark}
  Proceeding as in the proof of Proposition~\ref{prop:phi maps} we can use this result about the evaluation of the natural transformation $\varphi_{n, i, m}$ on a vector space $V$ of dimension at most $2n$ to deduce the following fact about the functorial description of $\coker(\varphi_{n, i, m})$: all diagrams appearing in the cokernel with nonzero multiplicity either have strictly more than~$2n$ rows, or strictly less than~$2n - i + m$ rows. It is plausible that no diagrams with strictly more than~$2n$ rows actually appear, but I don't know how to prove that.
\end{remark}

\begin{proof}
  First, note that the map $\varphi_{n, i, m}(V)$ is injective by Corollary~\ref{cor:injective phi maps}. Since $\varphi_{n, i, m}$ is a~$\GL(V)$-equivariant injective map, it is enough to prove that if~$\lambda$ is a diagram whose first column has $\geq 2n - i + m$ boxes, then it occurs in the decomposition of the source object and the target object with the same multiplicity. By Lemma~\ref{lem:mult is multiplicity} we need to prove that for any Young diagram $\lambda$ with $2n - i + 2m$ boxes whose first column's length is at least $2n - i + m$, we have an equality $\multl(\lambda, 2n-i) = \multl(\lambda, i)$.

  We explicitly construct a bijection between the two sets $\mathrm{diags}_{\Lambda-2\Z}(\lambda, -)$ for $i$ and $2n-i$. The idea is that the first column of any such $\lambda$ is much longer than the other columns, and within some bounds we can choose to remove or not to remove boxes from the first column independently of other columns. Explicitly, since the first column of~$\lambda$ has length at least~$2n - i + m$ and there are~$2n-i+2m$ boxes, we know that the length of the second column is at most
  \[
    2n-i+2m - (2n - i + m) = m,
  \]
  and hence the first column is longer than the second by at least
  \[
    2n - i + m - m = 2n - i.
  \]
  Thus, given any way in $\mathrm{diags}_{\Lambda-2\Z}(\lambda, i)$ to remove an $i$-long vertical strip from $\lambda$ so that all the columns of the remaining part have even lengths, we can always take $2n-2i$ more boxes from the first column, and this will be a valid Young diagram in $\mathrm{diags}_{\Lambda-2\Z}(\lambda, 2n-i)$; and vice versa, it is easy to check that any way to remove $(2n-i)$-long vertical strip from $\lambda$ requires taking at least $2n - 2i$ boxes from the first column, which we can decide not to remove.
\end{proof}

\begin{remark}
  The last part of the proof is essentially an application of the point (4) of Lemma~\ref{lem:even profile properties}, but since $\lambda$ is not necessarily a diagram with an even profile, giving a direct proof is less annoying than tracking the corrections related to the parity of the number $2n - i + m$.
\end{remark}

\section{Semiorthogonality on $X_1$}
\label{sec:lefschetz rectangular by computation}

In this section we prove the following theorem. We use the notation from Section~\ref{ssec:intro gr2n}.

\begin{theorem}
  \label{thm:lefschetz grouping}
  \begin{enumerate}
  \item The sequence $\langle L_0, L_1, \ldots, L_{n-1} \rangle$ is an exceptional collection on $X_1$.
  \item\label{thm:lefschetz rectangular part} Let $\mA_0$ be the category generated by the sequence above. Then the sequence
    \[
      \langle \mA_0, \mA_0(H), \ldots, \mA_0(k_{\mathrm{max}}H) \rangle
    \]
    is semiorthogonal where $k_{\mathrm{max}} = \binom{2n-1}{2} - 2$.
  \item\ Moreover, that semiorthogonal sequence of subcategories can be extended:
    \begin{equation}
      \label{eq:lefschetz decomposition}
      \langle \mA_0, \mA_0(H), \ldots, \mA_0(k_{\mathrm{max}}H), L_0((k_{\mathrm{max}} + 1)H), \ldots, L_0((k_{\mathrm{max}} + n)H) \rangle
    \end{equation}
    by adding $n$ copies of the twists of $L_0 \caniso \O$. We refer to the subcategory generated by this sequence as $\mA$.
  \end{enumerate}
\end{theorem}

The proof is not just a straightforward computation reducing everything to the Borel--Weil--Bott theorem: due to the nature of the objects $L_i$, which are defined as cones of some maps of vector bundles, in the proof of semiorthogonality we end up needing to show that certain morphisms between (reducible) $\GL$-representations are far from being zero, in that they restrict to isomorphisms on certain classes of irreducible summands of those representations. This is done by using the results from Section~\ref{sec:symmetric algebra exterior square}, specifically Corollary~\ref{cor:injective phi maps} and Lemma~\ref{lem:coker for small twists}.

Since the category $\mA$ from Theorem~\ref{thm:lefschetz grouping} is generated by the twists of $L_i$'s by some powers of the line bundle $\O(H)$, to prove the theorem we need to show the vanishing
\begin{equation}
  \label{eq:lefschetz rectangular semiorthogonality}
  \RHom_{X_1}(L_i, L_j(-kH)) = 0
\end{equation}
in three kinds of situations:
\begin{itemize}
\item if $k = 0$ and $n-1 \geq i > j \geq 0$ (this is the first claim of the theorem); or
\item if $k = 1, \ldots, \binom{2n-1}{2}-2$ and $i, j \in [0, n-1]$ are arbitrary (this is the second claim of the theorem).
\item if $k = \binom{2n-1}{2} - 1, \ldots, \binom{2n-1}{2} + n - 2$, when $i = 0$ and $j \in [0; n-1]$ is arbitrary (this is the third claim of the theorem).
\end{itemize}

Before proceeding with the proof, we will discuss some basic properties of the objects~$L_i$ and the morphism~$\pi\colon X_1 \to \P(V)$ in Subsection~\ref{ssec:semiorthogonality basic observations}. After that, we prove these three cases in three subsections, \ref{ssec:lis semiorthogonality}, \ref{ssec:rectangular semiorthogonality} and \ref{ssec:lefschetz tail semiorthogonality}, respectively. 

\subsection{Basic observations}
\label{ssec:semiorthogonality basic observations}

In this subsection we collect a few basic facts about objects~$L_i$ and the variety~$X_1$. Recall that the morphism~$\pi\colon X_1 \to \P(V)$ is the projectivization of the vector bundle~$\Lambda^2 Q^\dual$ on~$\P(V)$, and $\O(H)$ is the relative ample line bundle for the morphism~$\pi$. The line bundle~$\O(H)$ coincides with the pullback of $\O_{\P(\Lambda^2 V^\dual)}(1)$ along the map $p\colon X_1 \to \P(\Lambda^2 V^\dual)$. We also denote by~$\O(h)$ the line bundle $\pi^*\O_{\P(V)}(1)$.

\begin{lemma}
  \label{lem:serre duality on x1}
  The canonical bundle of $X_1$ is $\O(-\binom{2n-1}{2}H - 2h)$.
\end{lemma}
\begin{proof}
  Direct computation using the fact that $\pi\colon X_1 \to \P(V)$ is the projectivization of the vector bundle $\Lambda^2 Q^\dual$ of rank $\binom{2n-1}{2}$.
\end{proof}

\begin{lemma}
  \label{lem:pushforwards fiberwise}
  Let $k \in \Z$ be an integer. Then:
  \begin{enumerate}
  \item If $k \geq 0$, then $\pi_*\O(kH) \caniso \Sym^k(\Lambda^2Q)$.
  \item If $k$ is between $-1$ and $-(\binom{2n-1}{2} - 1)$, then $\pi_*\O(kH) = 0$.
  \item If $k = -(\binom{2n-1}{2} + k^\prime)$ with $k^\prime \geq 0$, then
    \[
      \pi_*\O(kH) \caniso \det(\Lambda^2 Q^\dual) \otimes \Sym^{k^\prime}(\Lambda^2 Q^\dual)[-d] \caniso \O(-(2n-2)h) \otimes \Sym^{k^\prime}(\Lambda^2 Q^\dual)[-d],
    \]
    where $d = \dim X_1$.
  \end{enumerate}
\end{lemma}
\begin{proof}
  As above, this is a general property of projectivizations of vector bundles.
\end{proof}

Recall that for an integer $i \in [0; n-1]$ the object $L_i \in \Dbcoh(X_1)$ is defined as the cone of the morphism
\[
  \pi^*\Lambda^{2n - i}Q \otimes \O(-(n-i)H) \to \pi^*\Lambda^i Q.
\]
given by the multiplication with the $(n-i)$'th power of the $2$-form corresponding to the tautological section of $\O(H)$. Note that $L_0$ is by definition isomorphic to $\O_{X_1}$, so it is an exceptional object.

\begin{lemma}
  \label{lem:li are exceptional}
  Let $i \in [1, n-1]$. 
  Then the following properties hold:
  \begin{enumerate}
  \item The objects $\langle \pi^*\Lambda^{2n-i}Q \otimes \O(-(n-i)H), \,\, \pi^*\Lambda^iQ \rangle$ form an exceptional pair.
  \item The graded vector space
    \[
      \RHom(\pi^*\Lambda^{2n-i}Q \otimes \O(-(n-i)H), \pi^*\Lambda^iQ)
    \]
    is one-dimensional, concentrated in degree $0$.
  \item The object $L_i$ is an exceptional sheaf.
  \end{enumerate}
\end{lemma}
\begin{proof}
  The bundle~$\pi^*\Lambda^iQ$ is exceptional since~$\Lambda^iQ$ is an exceptional bundle on~$\P(V)$ and~$\pi^*$ is a fully faithful functor. Similarly, the bundle $\pi^*\Lambda^{2n - i}Q$ and any its twists are exceptional. To show that they form an exceptional pair, we use the adjunction between $\pi^*$ and $\pi_*$:
  \[
    \begin{aligned}
    \RHom(\pi^*\Lambda^iQ, & \pi^*\Lambda^{2n - i}Q \otimes \O(-(n-i)H)) \caniso \\ & \caniso \RGamma(\P(V), \Lambda^i Q^\dual \otimes \Lambda^{2n-i}Q \otimes \pi_*\O(-(n-i)H)).
    \end{aligned}
  \]
  Here by Lemma~\ref{lem:pushforwards fiberwise} the object $\pi_*\O(-(n-i)H)$ is zero, so everything vanishes and the first claim of the lemma is proved.

  For the second claim, we compute  
  \[
    \begin{aligned}
      \RHom(\pi^*\Lambda^{2n-i}Q & \otimes \O(-(n-i)H), \pi^*\Lambda^iQ) \caniso \\
      & \caniso \RGamma(\P(V), \Lambda^{2n-i}Q^\dual \otimes \Lambda^i Q \otimes \pi_*\O((n-i)H)) \caniso \\
      & \caniso \RGamma(\P(V), \Lambda^{2n-i}Q^\dual \otimes \Lambda^i Q \otimes \Sym^{n-i}(\Lambda^2 Q)) \caniso \\
      & \caniso \RGamma(\P(V), \O(-1) \otimes \Lambda^{i-1}Q \otimes \Lambda^i Q \otimes \Sym^{n-i}(\Lambda^2 Q)),
    \end{aligned}
  \]
  where for the second isomorphism we used Lemma~\ref{lem:pushforwards fiberwise} and the third follows from the linear-algebraic observation that $\Lambda^a Q^\dual \iso \Lambda^{\rk(Q) - a} Q \otimes \det(Q^\dual)$ for any $a$. Consider the decomposition of $\Lambda^{i-1}Q \otimes \Lambda^i Q \otimes \Sym^{n-i}(\Lambda^2 Q)$ into a direct sum of Schur functors $\oplus_\lambda \Sigma^\lambda(Q)$. First, note that all diagrams $\lambda$ appearing in this decomposition have exactly
  \[
    i - 1 + i + 2 (n - i) = 2n - 1
  \]
  boxes, while the acyclicity part of the Borel--Weil--Bott theorem (Corollary~\ref{cor:bwb acyclicity}) implies that the cohomology~$\RGamma(\P(V), \O(-1) \otimes \Sigma^\lambda(Q))$ vanishes if $\lambda$ has strictly less than $2n-1$ rows. Thus the only diagram~$\lambda$ with nonzero contribution to the cohomology is the vertical column of height~$2n - 1$. By the Pieri's formula (Theorem~\ref{thm:pieri}) and the description of plethysm~$\Sym^\bullet(\Lambda^2(-))$ (Theorem~\ref{thm:plethysm description}) it is easy to see that this diagram appears with multiplicity $1$, and the Borel--Weil--Bott theorem~\ref{thm:bwb} then shows that the cohomology is one-dimensional, in degree zero, as expected.

  The first two claims imply that $L_i$ is an exceptional object since it becomes identified with the mutation of an exceptional pair. For the last claim it remains to show that $L_i$ is a sheaf, not a complex. By definition it is enough to prove that the morphism
  \begin{equation}
    \label{eq:definition of Li}
    \pi^*\Lambda^{2n - i}Q \otimes \O(-(n-i)H) \to \pi^*\Lambda^i Q.
  \end{equation}
  is injective. Since this is a map between vector bundles, it is enough to prove the injectivity of the map of fibers on a general point. By the definition of $X_1$ a general point corresponds to a pair: a filtration
  \[
    0 \subset K_1 \subset K_2 \subset V
  \]
  of $V$ by subspaces with $\dim K_1 = 1$ and $\dim K_2 = 2$, together with a nondegenerate $2$-form~$\omega$ on~$V/K_2$. After choosing a splitting $V \iso V/K_2 \oplus K_2$, the map induced by \eqref{eq:definition of Li} on the fibers at that point can be described in these terms as the morphism
  \begin{equation}
    \label{eq:fiberwise map Li on general point}
    \Lambda^{2n - i}(V/K_2 \oplus K_2/K_1) \xrightarrow{\omega^{n-i}} \Lambda^i(V/K_2 \oplus K_2/K_1),
  \end{equation}
  where $\omega$ does not act on the one-dimensional vector space~$K_2/K_1$. Decomposing the exterior powers of a direct sum into sums of exterior powers we see that the map \eqref{eq:fiberwise map Li on general point} splits into a direct sum of two maps: first, the map
  \[
    \Lambda^{2n-i}(V/K_2) \xrightarrow{\omega^{n-i}} \Lambda^i(V/K_2),
  \]
  and this is an injective map since its composition with another substitution of $\omega$ becomes the pairing of complementary exterior powers
  \[
    \Lambda^{2n-2 - (i - 2)} (V/K_2) \xrightarrow{\omega^{n-i + 1}} \Lambda^{i-2}(V/K_2),
  \]
  which is an isomorphism as $\omega$ is a nondegenerate $2$-form on a $(2n-2)$-dimensional vector space~$V/K_2$ (note that $i \leq n-1$ by assumption; for $i=1$ we get an isomorphism of zero vector spaces). The second summand of~\eqref{eq:fiberwise map Li on general point} is the twist by the one-dimensional vector space $K_2/K_1$ of the map
  \[
    \Lambda^{2n-2 - (i - 1)}(V/K_2) \xrightarrow{\omega^{n-i}} \Lambda^{i-1}(V/K_2),
  \]
  which is again a nondegenerate pairing of complementary exterior powers of $V/K_2$.
\end{proof}

Another fact that we use in the computations below is the following identification. It is almost tautological, but I don't know a concise way of explaining it (beyond saying ``it is easy to see''), so the proof is lengthier than it has any right to be.

\begin{lemma}
  \label{lem:lis and tautological morphisms}
  Let $0 \leq i \leq n-1$ be an index for the object $L_i \in \Dbcoh(X_1)$, and let $k \in \Z$ be an integer.
  \begin{enumerate}
  \item If $k \geq n - i$, then by Lemma~\textup{\ref{lem:pushforwards fiberwise}} we have
    \[
      \begin{aligned}
        \pi_*(L_i(kH)) & \caniso [ \Lambda^{2n-i} Q \otimes \pi_*\O((k-(n-i))H) \to \Lambda^{i} Q \otimes \pi_*\O(kH) ] \caniso \\
        & \caniso [ \Lambda^{2n-i} Q \otimes \Sym^{k - n + i}(\Lambda^2 Q) \xrightarrow{\varphi} \Lambda^{i} Q \otimes \Sym^{k}(\Lambda^2 Q) ],
      \end{aligned}
    \]
    and the morphism $\varphi$ is equal to the map $\varphi_{n, i, k - n + i}(Q)$ from Proposition~\textup{\ref{prop:phi maps}}.
  \item If $k = -(\tbinom{2n-1}{2} + k^\prime)$ with $k^\prime \geq 0$, then by Lemma~\textup{\ref{lem:pushforwards fiberwise}} and the linear-algebraic identity~$\Lambda^a Q \caniso \Lambda^{2n - 1 - a} Q^\dual \otimes \O(1)$ the pushforward~$\pi_*(L_i(kH))$ is isomorphic to the object
    \[
      \begin{aligned}
        [
        \Lambda^{2n-i} &  Q \otimes \O(-(2n-2)) \otimes \Sym^{k^\prime + n - i}(\Lambda^2 Q^\dual)
        \to
        \Lambda^i Q \otimes \O(-(2n-2)) \otimes \Sym^{k^\prime}(\Lambda^2 Q^\dual)
        ] \caniso \\
        & \caniso \O(-(2n-3)) \otimes
        [
        \Lambda^{i-1} Q^\dual \otimes \Sym^{k^\prime + n - i}(\Lambda^2 Q^\dual)
        \xrightarrow{\psi}
        \Lambda^{2n - 2 - (i-1)} Q^\dual \otimes \Sym^{k^\prime}(\Lambda^2 Q^\dual)
        ],
      \end{aligned}
    \]
    and the morphism $\psi$ is equal to the dual of the map $\varphi_{n-1, i-1, k^\prime}(Q)$ from Proposition~\textup{\ref{prop:phi maps}}.
  \end{enumerate}
\end{lemma}

\begin{proof}
  To see the compatibility, we first define the ``affine'' versions of $X_1$ and the objects~$L_i$'s as follows. Denote by~$X_{1, \aff}$ the total space~$\Tot_{\P(V)}(\Lambda^2 Q^\dual)$ and let~$\pi_{\aff}\colon X_{1, \aff} \to \P(V)$ be the projection morphism. Then we set $L_{i, \aff}$ to be the cone of the morphism
  \begin{equation}
    \label{eq:affine li}
    \pi_{\aff}^* \Lambda^{2n-i} Q \to \pi_{\aff}^* \Lambda^i Q
  \end{equation}
  of vector bundles on $X_{1, \aff} \caniso \Tot_{\P(V)}(\Lambda^2 Q^\dual)$ given by the multiplication by the $(n-i)$'th power of the $2$-form on $Q$.
  % The same argument that was used to prove that $L_i$ is a coherent sheaf on $X_1$ in Lemma~\ref{lem:li are exceptional} shows that the morphism~\eqref{eq:affine li} is injective, so $L_{i, \aff}$ is also a coherent sheaf.

  Next, we discuss the relation between $L_i$ and $L_{i, \aff}$, as well as their pushforwards to $\P(V)$ along the maps $\pi$ and $\pi_{\aff}$, respectively. Consider the~$\C^\star$-action by scaling on the vector space~$V$. It induces a natural~$\C^\star$-equivariant structure on the vector bundle~$Q$ on~$\P(V)$, thus on~$\Lambda^2 Q$ as well, which in turn induces a fiberwise~$\C^\star$-action on~$X_{1, \aff}$.
  Moreover, both pullbacks~$\pi_{\aff}^* \Lambda^{2n-i} Q$ and~$\pi_{\aff}^* \Lambda^i Q$ on~$X_1$ can be equipped with the natural~$\C^\star$-equivariant structure and the morphism~\eqref{eq:affine li} is~$\C^\star$-invariant. By definition of the projectivization we know that a global section of, say, some twist of the bundle~$\pi^* \Lambda^i Q$ by a positive power of the relative ample divisor on~$X_1 := \P_{\P(V)}(\Lambda^2 Q^\dual)$ is the same thing as a~$\C^\star$-equivariant global section (with some specific weight) of the restriction of the bundle~$\pi_{\aff}^* \Lambda^i Q$ to the punctured variety~$X_{1, \aff} \setminus \P(V)$, where we remove the zero section from the total space of the vector bundle. Note that since $X_{1, \aff}$ is a smooth variety and the zero section $\P(V) \subset X_{1, \aff}$ is not a divisor, a global section of $\pi_{\aff}^* \Lambda^i Q$ on the punctured total space always extends to a global section on the whole $X_1$ by Hartogs's extension theorem. Finally, note that the same holds not only for global sections, but also on each fiber of the projection to $\P(V)$, separately.

  In other words, what we have established above is that for $k \geq n-i$ the morphism
  \begin{equation}
    \label{eq:nonaffine li pushforward}
    \Lambda^{2n-i} Q \otimes \pi_*\O((k - (n-i))H) \to \Lambda^i Q \otimes \pi_*\O(kH)
  \end{equation}
  whose cone is isomorphic to the pushforward $\pi_*(L_i(kH))$ is a direct summand of the morphism on $\P(V)$ that we get by pushing forward the map~\eqref{eq:affine li} along the affine morphism $\pi_{\aff}$:
  \begin{equation}
    \label{eq:affine li pushforward}
    \Lambda^{2n-i} Q \otimes \pi_{\aff, *}(\O_{X_{1, \aff}}) \to \Lambda^i Q \otimes \pi_{\aff, *}(\O_{X_{1, \aff}}).
  \end{equation}
  In fact, the map~\eqref{eq:nonaffine li pushforward} coincides with one of the summands in the decomposition of the morphism~\eqref{eq:affine li pushforward} with respect to the weights of the $\C^\star$-action.

  The variety $X_{1, \aff} := \Tot_{\P(V)}(\Lambda^2 Q^\dual)$ is by definition a relative spectrum of the graded algebra object $\Sym^\bullet(\Lambda^2 Q)$ on $\P(V)$. Coherent sheaves on $X_{1, \aff}$ correspond to module objects over that algebra, and the correspondence is realized by the pushforward along the projection morphism $\pi_{\aff}$. Thus the pushforward of the map of pullbacks~\eqref{eq:affine li} corresponds to some morphism of free (graded) modules
  \[
    \Lambda^{2n - i} Q \otimes \Sym^{\bullet - (n-i)}(\Lambda^2 Q) \to \Lambda^i Q \otimes \Sym^{\bullet}(\Lambda^2 Q).
  \]
  By looking at a single fiber of the projection map $\Tot_{\P(V)}(\Lambda^2 Q^\dual) \to \P(V)$ separately it is easy to check that the map of free modules corresponding to the tautological multiplication by the $(n-i)$'th power of the $2$-form on $Q$ is the application of the comultiplication
  \[
    \Lambda^{2n-i} Q \to \Lambda^i Q \otimes (\Lambda^2 Q)^{\otimes (n-i)}
  \]
  in the exterior algebra and then the multiplication in the symmetric algebra $\Sym^{\bullet}(\Lambda^2 Q)$. So this morphism coincides with the direct sum of maps $\varphi_{n, i, m}(Q)$ defined in Proposition~\ref{prop:phi maps} over all values of the integer $m$. And this is exactly what we needed to prove for the first part of the lemma.

  For the second part, note that the dual to the object $\pi_*(L_i(kH))$ is isomorphic to
  \[
    \begin{split}
      \pi_*(L_i(kH))^\dual & := \intRHom_{\P(V)}(\pi_*(L_i(kH)), \O_{\P(V)}) \caniso \pi_* \intRHom_{X_1}(L_i(kH), \pi^! \O_{\P(V)}) \caniso \\
      & \caniso \pi_*(L_i^\dual \otimes \O(-kH) \otimes \omega_{X_1} \otimes \omega_{\P(V)}^\dual) \caniso \\
      & \caniso \pi_*(L_i^\dual(-(k + \tbinom{2n-1}{2})H)) \otimes \O((2n-2)h).
    \end{split}
  \]
  where we used Grothendieck duality for the proper morphism $\pi$ for the first isomorphism and Lemma~\ref{lem:serre duality on x1} for the second one. Here the twist $-(k + \tbinom{2n-1}{2})$ is a non-negative number by assumption. Now we can use an argument analogous to the first part to compute the dual to the object $\pi_*(L_i(kH))$ since we only need to compute the pushforward of objects which fiberwise have no higher cohomology.
\end{proof}

\subsection{Semiorthogonality of $L_i$'s}
\label{ssec:lis semiorthogonality}

In this subsection we prove that the sequence
\[
  \langle L_0, \ldots, L_{n-1} \rangle
\]
is an exceptional collection in~$\Dbcoh(X_1)$. In fact, we prove a somewhat stronger statement: each $L_i$ is defined as a cone of a morphism between two vector bundles, and we show that for~$i < j$ the vector bundles used to define~$L_i$ are pairwise right-orthogonal to the vector bundles used to define~$L_j$. This is a basic computation using Borel--Weil--Bott theorem (Theorem~\ref{thm:bwb}). We use the following combinatorial lemma.

\begin{lemma}
  \label{lem:orthogonal to exterior powers}
  Let $0 \leq m \leq \rk(Q)$ be a number, and let $\lambda$ be a Young diagram with strictly less than~$m$ rows. Then~$\RHom_{\P(V)}(\Lambda^m Q, \Sigma^{\lambda}(Q)) = 0$. Moreover, for any number $m^\prime \geq 0$ we have the vanishing
  \[
    \RHom_{\P(V)}(\Lambda^{m+m^\prime}Q \otimes (\det Q)^{\otimes m^\prime}, \Sigma^{\lambda}(Q)) = 0.
  \]
\end{lemma}
\begin{proof}
  For $m^\prime = 0$ note that
  \[
    \RHom(\Lambda^m Q, \Sigma^{\lambda}(Q)) \caniso \RGamma(\Lambda^m Q^\dual \otimes \Sigma^\lambda Q) \caniso \RGamma(\det Q^\dual \otimes \Lambda^{2n-1-m}Q \otimes \Sigma^\lambda Q).
  \]
  The tensor product $\Lambda^{2n-1-m}Q \otimes \Sigma^\lambda Q$ can be decomposed into a direct sum of Schur functors of $Q$, of some weights. By Corollary~\ref{cor:bwb acyclicity} we know that if a diagram $\mu$ has strictly less than~$2n-1$ rows, then~$\RGamma(\det Q^\dual \otimes \Sigma^{\mu}(Q)) = 0$. However, since $\lambda$ is assumed to have $< m$ rows, by the Pieri rule any Young diagrams occuring in the decomposition has $< 2n-1$ rows, and the lemma follows. The case of nonzero $m^\prime$ is very similar.
\end{proof}

\begin{remark}
  In this subsection we only need the $m^\prime = 0$ case of the lemma. The more general statement is useful later in the paper.
\end{remark}

Now we can proceed with the proof.

\begin{lemma}
  \label{lem:lis semiorthogonality}
  Let $0 \leq i < j \leq n-1$ be two indices. Then $\RHom_{X_1}(L_j, L_i) = 0$.
\end{lemma}
\begin{proof}
  By definition we have
  \[
    \begin{aligned}
      L_i & := [ \pi^*\Lambda^{2n-i} Q \otimes \O(-(n-i)H) \to \pi^* \Lambda^i Q ] \\
      L_j & := [ \pi^*\Lambda^{2n-j} Q \otimes \O(-(n-j)H) \to \pi^* \Lambda^j Q ].
    \end{aligned}
  \]
  
  Let us consider the pairwise maps between vector bundles appearing in the definitions of~$L_i$ and~$L_j$. First, we compute
  \[
    \begin{split}
      \RHom_{X_1}(\pi^*\Lambda^j Q, & \pi^* \Lambda^{2n-i} Q \otimes \O(-(n-i)H)) \caniso \\
      & \RHom_{\P(V)}(\Lambda^j Q, \Lambda^{2n-i} Q \otimes \pi_*\O(-(n-i)H)).
    \end{split}
  \]
  Since $i \in [0; n-1]$, the pushforward $\pi_*\O(-(n-i)H)$ is a zero object by Lemma~\ref{lem:pushforwards fiberwise}, so the graded space above is zero. Next, consider
  \[
    \RHom_{X_1}(\pi^* \Lambda^j Q, \pi^* \Lambda^i Q) \caniso \RHom_{\P(V)}(\Lambda^j Q, \Lambda^i Q).
  \]
  Since by assumption $j > i$ this vanishes by Lemma~\ref{lem:orthogonal to exterior powers}. Next, consider the space
  \[
    \begin{split}
      \RHom_{X_1}(\pi^* \Lambda^{2n-j}Q \otimes \O(-(n-j)H) &, \pi^* \Lambda^{2n-i} Q \otimes \O(-(n-i)H)) \caniso \\
      & \caniso \RHom_{\P(V)}(\Lambda^{2n-j}Q, \Lambda^{2n-i}Q \otimes \pi_*\O((i-j)H)).
    \end{split}
  \]
  Again, since $j > i$ the number $i - j$ is a negative number of magnitude at most $n-1$ and the pushforward $\pi_*\O((i-j)H)$ is zero by Lemma~\ref{lem:pushforwards fiberwise}. Finally, consider the graded space
  \[
    \begin{split}
      \RHom_{X_1}(\pi^* \Lambda^{2n-j} Q & \otimes \O(-(n-j)H), \pi^* \Lambda^i Q) \caniso \\
      & \caniso \RHom_{\P(V)}(\Lambda^{2n-j} Q, \Lambda^i Q \otimes \pi_*\O((n-j)H)).
    \end{split}
  \]
  By Lemma~\ref{lem:pushforwards fiberwise} the pushforward $\pi_*\O((n-j)H)$ is isomorphic to $\Sym^{n-j}(\Lambda^2 Q)$. So we need to show the vanishing of the space
  \[
    \RHom_{\P(V)}(\Lambda^{2n-j}Q, \Lambda^i Q \otimes \Sym^{n-j}(\Lambda^2 Q)).
  \]
  Consider the target vector bundle. It can be decomposed into a direct sum of Schur functors of $Q$ with positive weights, and the number of boxes in each appearing diagram is equal to
  \[
    i + 2(n-j) = 2n - j - (j - i),
  \]
  which is strictly smaller than $2n-j$ since $i < j$ by assumption. Thus, in particular, any appearing diagram has strictly less than $2n-j$ rows, and by Lemma~\ref{lem:orthogonal to exterior powers} all such diagrams lie in the orthogonal to $\Lambda^{2n-j} Q$, which confirms the semiorthogonality.

  The four semiorthogonalities together imply that $L_i$ is right-orthogonal to $L_j$, which is exactly what we needed to prove.
\end{proof}

\subsection{Semiorthogonality of the rectangular part}
\label{ssec:rectangular semiorthogonality}

This subsection is dedicated to the proof of the second claim of Theorem~\ref{thm:lefschetz grouping}, i.e., to the fact that the subcategories~$\mA_0$ and~$\mA_0(k H)$ are semiorthogonal for~$k$ between~$1$ and~$\binom{2n-1}{2} - 2$. For the ease of reading we use the following definition.

\begin{definition}
  A number $k \in \Z$ is said to be \emph{in the rectangular range} if $k \in [1; \binom{2n-1}{2} - 2]$.
\end{definition}

In this subsection we generally assume that $k$ is in the rectangular range. Additionally, we will indicate what happens for the case $k = 0$ in remarks below the main computations.

Since by definition $\mA_0$ is generated by objects~$L_i$ for~$i \in [0; n-1]$, the desired semiorthogonality of subcategories is equivalent to the fact that for~$i, j \in [0, n-1]$ the objects~$L_i(kH)$ and~$L_j$ are semiorthogonal. Some computations below depend on which of two numbers,~$i$ or~$j$, is larger. Instead of repeating more or less the same computation two times, we use the following observation:

\begin{lemma}
  \label{lem:ij symmetry with twist}
  Let $i \leq j$ be two numbers. The following statements are equivalent:
  \begin{enumerate}
  \item The vanishing
    \[
      \RHom(L_j, L_i(-k H)) = 0
    \]
    holds for any $k$ in the rectangular range.
  \item Let $\O(\delta)$ be the line bundle $\O_{X_1}(-H - 2h)$. The vanishing
    \[
      \RHom(L_i, L_j (-k H) \otimes \O(\delta)) = 0
    \]
    holds for any $k$ in the rectangular range.
  \end{enumerate}
\end{lemma}
\begin{proof}
  Using Serre duality in $\Dbcoh(X_1)$ and Lemma~\ref{lem:serre duality on x1} we compute
  \[
    \begin{aligned}
      \RHom(L_j, & L_i (-k H))^\dual[\dim X_1] \caniso \RHom(L_i(-kH), L_j \otimes \omega_{X_1}) \caniso \\
      & \caniso \RHom(L_i, L_j((k - \tbinom{2n-1}{2})H - 2h)) \caniso \\
      & \caniso \RHom(L_i, L_j(-(\tbinom{2n-1}{2} -k  - 1)H) \otimes \O(-H - 2h)).
    \end{aligned}
  \]
  Observe that the correspondence~$k \mapsto \tbinom{2n-1}{2} - k - 1$ defines a bijection from the rectangular range~$[1; \binom{2n-1}{2} - 2]$ to itself. The lemma follows.
\end{proof}

By Lemma~\ref{lem:ij symmetry with twist} the second claim of Theorem~\ref{thm:lefschetz grouping} is equivalent to the following lemma.

\begin{lemma}
  \label{lem:ilessj case}
  Let $0 \leq i \leq j \leq n-1$ be two indices, and let $k \in \Z$ be in the rectangular range. Then $\RHom(L_i, L_j(-kH)) = 0$. In fact, the object $L_j(-kH)$ is right-orthogonal to the exceptional pair
  \begin{equation}
    \label{eq:exceptional pair for Li}
    \langle \pi^* \Lambda^{2n-i}Q \otimes \O(-(n-i)H), \Lambda^iQ \rangle
  \end{equation}
  used to define $L_i$. Moreover, the same is true after replacing $L_j(-kH)$ with $L_j(-kH) \otimes \O(\delta)$, where $\delta = -H - 2h$ as defined in Lemma~\textup{\ref{lem:ij symmetry with twist}}.
\end{lemma}

We prove the semiorthogonality to the two objects in~\eqref{eq:exceptional pair for Li} separately, in lemmas~\ref{lem:ilessj easy part} and~\ref{lem:semiorthogonality ilessj key part}.

\begin{lemma}
  \label{lem:ilessj easy part}
  Let $0 \leq i \leq j \leq n-1$ be two indices, and let $k$ be in the rectangular range. Then the object $\Lambda^i Q$ is left-orthogonal to the exceptional pair
  \begin{equation}
    \label{eq:exceptional pair for Lj twist}
    \langle \pi^* \Lambda^{2n-j}Q \otimes \O(-(n-j+k)H), \Lambda^jQ \otimes \O(-kH) \rangle
  \end{equation}
  and hence to the object $L_j(-kH)$ as well. Moreover, the same is true after replacing $L_j(-kH)$ with $L_j(-kH) \otimes \O(\delta)$, where $\delta = -H - 2h$ as defined in Lemma~\textup{\ref{lem:ij symmetry with twist}}.
\end{lemma}

\begin{remark}
  This is an example of a statement that works only under the assumption that~$i \leq j$. For~$i > j$ the object~$\Lambda^i Q$ is not, in general, left-orthogonal to the pair~\eqref{eq:exceptional pair for Lj twist}. Without Lemma~\ref{lem:ij symmetry with twist} we would have to make a separate, but very similar statement about the case~$i > j$.
\end{remark}

\begin{proof}
  First observe that
  \[
    \RHom_{X_1}(\pi^* \Lambda^i Q, \pi^* \Lambda^jQ \otimes \O(-kH)) = \RHom_{\P(V)}(\Lambda^iQ, \Lambda^j Q \otimes \pi_*\O(-kH)).
  \]
  By Lemma~\ref{lem:pushforwards fiberwise} the pushforward $\pi_*\O(-kH)$ vanishes for any $k$ from $1$ to $\binom{2n-1}{2}-1$. Since $k$ is in the rectangular range, it is between $1$ and $\binom{2n-1}{2} - 2$, hence the space vanishes and $\Lambda^i Q$ is semiorthogonal to the second object of the exceptional pair~\eqref{eq:exceptional pair for Li}. This argument shows that the space vanishes not only for the twist by $\O(-kH)$, but for $\O(-kH) \otimes \O(\delta)$ as well.

  Now consider semiorthogonality to the first object in~\eqref{eq:exceptional pair for Lj twist}:
  \[
    \begin{split}
      \RHom_{X_1}(\pi^* \Lambda^i Q, & \pi^* \Lambda^{2n-j}Q \otimes \O(-(n-j+k)H)) \caniso \\
      & \caniso \RHom_{\P(V)}(\Lambda^iQ, \Lambda^{2n-j}Q \otimes \pi_*\O(-(n-j+k)H)).
    \end{split}
  \]
  Since $n-j+k$ is always positive (as $j \leq n-1$), for small values of $k$ the pushforward is also zero. But for $k = \binom{2n-1}{2} - (n-j)$ and for larger values of $k$ the pushforward is not zero. According to Lemma~\ref{lem:pushforwards fiberwise} we will get
  \[
    \begin{array}{>{\displaystyle}l c >{\displaystyle}r}
      k = \binom{2n-1}{2} - (n-j) & \rightsquigarrow & \O(-(2n-2))[\ldots] \\
      k = \binom{2n-1}{2} - (n-j-1) & \rightsquigarrow & \O(-(2n-2))[\ldots] \otimes \Lambda^2 Q^\dual \\
      \cdots & \rightsquigarrow & \cdots \\
      k = \binom{2n-1}{2} -2 & \rightsquigarrow & \O(-(2n-2))[\ldots] \otimes \Sym^{n - j - 2}(\Lambda^2 Q^\dual)
    \end{array}
  \]
  Thus to prove the semiorthogonality for every $k$ in the rectangular range it remains to show the vanishing of
  \[
    \begin{split}
      \RHom(\Lambda^i Q, & \Lambda^{2n-j} Q \otimes \O(-(2n-2)) \otimes \Sym^{h}(\Lambda^2 Q^\dual)) \caniso \\
      \caniso & \, \RGamma(\P(V), \Lambda^i Q^\dual \otimes \Lambda^{2n-j} Q \otimes \O(-(2n-2)) \otimes \Sym^h(\Lambda^2 Q^\dual))
    \end{split}
  \]
  where $h \in [0, n-j-2]$. Let us write this differently, in the form where all bundles could be obtained by applying a Schur functor with a \emph{positive weight} to $Q^\dual$, by replacing the exterior power~$\Lambda^{2n-j} Q$ with~$\Lambda^{j-1} Q^\dual \otimes \O(1)$:
  \begin{equation}
    \label{eq:ilessj easy part}
    \RGamma(\P(V),
    \O(-(2n-3))
    \otimes
    \Lambda^i Q^\dual \otimes \Lambda^{j-1} Q^\dual \otimes \Sym^h(\Lambda^2 Q^\dual))
  \end{equation}
  
  It turns out that all Young diagram that occur in the Littlewood--Richardson decomposition of the tensor product on the right hand side have
  \[
    \leq i + j - 1 + 2(n - j - 2) = (2n - 3) - (j - i + 2)
  \]
  boxes. Since by assumption $i \leq j$, this number is strictly smaller than $2n-3$, and this implies that the cohomology~\eqref{eq:ilessj easy part} vanishes by Borel--Weil--Bott theorem (see Corollary~\ref{cor:bwb acyclicity}).

  Observe that replacing $\O(-kH)$ with~$\O(-kH) \otimes \O(-\delta)$, i.e., with the line bundle
  \[
    \O(-(k+1)H - 2h),
  \]
  results in the following modification: in the equation~\eqref{eq:ilessj easy part}, instead of $\O(-(2n-3))$, the negative twist is now~$\O_{\P(V)}(-(2n-1))$, while~$h$ lies in the segment $[1, n - j - 1]$. The same argument proves the semiorthogonality in this case.
\end{proof}

\begin{remark}
  One can track the case $k = 0$ through the computations to see that
  \[
    \RHom_{X_1}(\Lambda^i Q, L_j) \caniso \RHom_{\P(V)}(\Lambda^i Q, \Lambda^j Q) \caniso \Lambda^{j-i} V.
  \]
\end{remark}

Now that have proved Lemma~\ref{lem:ilessj easy part}, it remains to show that the first object in the exceptional pair~\eqref{eq:exceptional pair for Li} defining $L_i$ is semiorthogonal to $L_j(-kH)$. 

\begin{lemma}
  \label{lem:semiorthogonality ilessj key part}
  Let $0 \leq i \leq j \leq n-1$ be two indices, and let $k$ be in the rectangular range. Then the object $\pi^* \Lambda^{2n-i} Q \otimes \O(-(n-i))$ is semiorthogonal to $L_j(-kH)$. Moreover, the same is true after replacing $L_j(-kH)$ with $L_j(-kH) \otimes \O(\delta)$, where $\delta = -H - 2h$ as defined in Lemma~\textup{\ref{lem:ij symmetry with twist}}.
\end{lemma}
\begin{proof}
  We first discuss the case without the twist by $\O(\delta)$. We need to show that
  \[
    \begin{split}
      \RHom_{X_1}(& \pi^*\Lambda^{2n-i} Q, [\pi^* \Lambda^{2n-j} Q \otimes \O((j-i-k)H) \to \pi^* \Lambda^j Q ((n-i-k)H)] ) \caniso \\
      \caniso & \,\, \RHom_{\P(V)}(\Lambda^{2n-i}Q, \\
      & \qquad [\Lambda^{2n-j} Q \otimes
      \eqnmarkbox[blue]{firstpushforward}{\pi_*\O((j-i-k)H)}
      \to \Lambda^jQ \otimes
      \eqnmarkbox[red]{secondpushforward}{\pi_*\O((n-i-k)H)}])
      = 0
    \end{split}
  \]
  \annotate{below, left}{firstpushforward}{\protect\smallcircled{$1$}}
  \annotate{below, right}{secondpushforward}{\protect\smallcircled{$2$}}
  \vspace{0.3cm}

  Note that $n-i > j-i \geq 0$. Thus, depending on the value of $k$, there are three options for the signs of the twists \textcolor{blue}{\smallcircled{$1$}} and \textcolor{red}{\smallcircled{$2$}}. We will discuss them separately.

  If $k > n - i$, then both twists are negative numbers of magnitude $\leq k$. Since $k$ is in the rectangular range, by Lemma~\ref{lem:pushforwards fiberwise} both pushforwards are zero and the semiorthogonality is trivial.

  If $j-i < k \leq n-i$, then the twist \textcolor{blue}{\smallcircled{$1$}} is negative, while the twist \textcolor{red}{\smallcircled{$2$}} is non-negative. Thus to prove the semiorthogonality we need to show that
  \[
    \RHom(\Lambda^{2n-i}Q, \Lambda^j Q \otimes \pi_*\O(n-i-k))
    \caniso
    \RHom(\Lambda^{2n-i}Q, \Lambda^j Q \otimes \Sym^{n-i-k}(\Lambda^2 Q)) = 0
  \]
  in the specified range of $k$. As in the proof of Lemma~\ref{lem:lis semiorthogonality}, note that any Young diagram appearing in the decomposition of the target vector bundle into Schur functors of~$Q$ has~$j + 2(n-i-k)$ boxes, and since in this case~$j-i < k$ and $i \leq j$, we know that each diagram has strictly less than $2n-i$ boxes, in particular strictly less than $2n-i$ rows, so by Lemma~\ref{lem:orthogonal to exterior powers} all such diagrams lies in the orthogonal to $\Lambda^{2n-i}Q$, which confirms the semiorthogonality in this range of values for $k$.

  Finally, we come to the most difficult case, $k \leq j - i$. Then both twists \textcolor{blue}{\smallcircled{$1$}} and \textcolor{red}{\smallcircled{$2$}} are non-negative, and we want to show that the object $\Lambda^{2n-i}Q$ is semiorthogonal to the cone of the morphism $\pi_*(L_j((n - i -k)H))$:
  \begin{equation}
    \label{eq:doubly positive pushforward}
    \Lambda^{2n-j}Q \otimes \Sym^{j-i-k}(\Lambda^2 Q) \to \Lambda^jQ \otimes \Sym^{n-i-k}(\Lambda^2 Q).
  \end{equation}

  By Lemma~\ref{lem:lis and tautological morphisms} this is exactly the map $\varphi_{n, j, j-i-k}$ from Lemma~\ref{lem:coker for small twists}. So it is injective and its cokernel is a direct sum of Schur functors of Young diagrams with strictly less than~$(2n-i)-k$ rows. Hence by Lemma~\ref{lem:orthogonal to exterior powers} the cokernel is semiorthogonal to $\Lambda^{2n-i}Q$ since $k \geq 0$. This is exactly what we wanted to show.

  Now let us note what changes after replacing $L_j(-kH)$ by $L_j(-kH) \otimes
  \O(\delta)$. The vanishing we need to prove is now the following:
  \begin{equation}
    \label{eq:ilessj complicated with twist}
    \begin{split}
      & \RHom_{\P(V)} (\Lambda^{2n-i}Q \otimes \O(2), \\
      & \qquad [
      \Lambda^{2n-j}Q \otimes \pi_*\O((j-i-k-1)H)
      \to
      \Lambda^jQ \otimes \pi_*\O((n-i-k-1)H)
      ]) = 0.
    \end{split}
  \end{equation}
  Again, we divide the proof into three cases depending on the signs of the twists. When both twists are negative, they are negative with magnitude at most $k+1$, and then by Lemma~\ref{lem:pushforwards fiberwise} both pushforwards are zero. When the first twist is negative and the second is non-negative, i.e., when $k$ belongs to the segment $[j-i, n-i-1]$, we need to prove the vanishing
  \[
    \RHom(\Lambda^{2n-i}Q \otimes \O(2), \Lambda^jQ \otimes \Sym^{n-i-k-1}(\Lambda^2 Q)) = 0.
  \]
  As above, by Lemma~\ref{lem:orthogonal to exterior powers} it is enough to show that
  \[
    j + 2(n-i-k-1) < 2n - i - 2.
  \]
  Equivalently, we need to show that $2k - (j-i)$ is a strictly positive number. Note that we have two assumptions on $k$: first, $k$ is in the rectangular range, in particular $k \geq 1$; second, we currently consider the case where the sign of the first twist in~\eqref{eq:ilessj complicated with twist} is negative, which means $k \geq j - i$. The inequality follows.

  In the third case, where both twists in~\eqref{eq:ilessj complicated with twist} are non-negative, we need to prove that the object $\Lambda^{2n-i}Q \otimes \O(2)$ is left-orthogonal to the cone of the morphism
  \[
    \Lambda^{2n-j}Q \otimes \Sym^{j-i-k-1}(\Lambda^2 Q) \to \Lambda^jQ \otimes \Sym^{n-i-k-1}(\Lambda^2 Q).
  \]
  As above, this is the map $\varphi_{n, j, j-i-k-1}$ from Lemma~\ref{lem:coker for small twists}. It is injective, its cokernel splits into a direct sum of Schur functors of Young diagrams with strictly less than $(2n-i) - k - 1$ rows. Since $k$ is in the rectangular range, we have $k \geq 1$, so by Lemma~\ref{lem:orthogonal to exterior powers} the cone of this map is semiorthogonal to $\Lambda^{2n-i}Q \otimes \O(2)$.
\end{proof}

\begin{remark}
  One can track the case $k = 0$ through the proof to see that the vanishing
  \[
    \RHom_{X_1}(\pi^*\Lambda^{2n-i}Q \otimes \O(-(n-i)H), L_j) = 0
  \]
  holds.
\end{remark}

Lemmas~\ref{lem:ilessj easy part} and~\ref{lem:semiorthogonality ilessj key part} together imply Lemma~\ref{lem:ilessj case}, and, by Lemma~\ref{lem:ij symmetry with twist}, we see that for any~$k$ in the rectangular range and any $i, j$ in the range $[0; n-1]$ we have the semiorthogonality
\[
  \RHom_{X_1}(L_i, L_j(-kH)) = 0.
\]
In other words, we proved the second claim of Theorem~\ref{thm:lefschetz grouping}.

Additionally, patching together the remarks after Lemmas~\ref{lem:ilessj easy part} and~\ref{lem:semiorthogonality ilessj key part} concerning the extra case $k=0$, one sees the following enhancement to the results of Subsection~\ref{ssec:lis semiorthogonality}:
\begin{lemma}
  Let $i \leq j \in [0; n-1]$ be two indices.
  Then $\RHom_{X_1}(L_i, L_j) \caniso \Lambda^{j-i} V$. For a triple~$i \leq j \leq k$ the composition map
  \[
    \RHom(L_i, L_j) \otimes \RHom(L_j, L_k) \to \RHom(L_i, L_k)
  \]
  is isomorphic to the multiplication
  \[
    \Lambda^{j-i} V \otimes \Lambda^{k-j} V \to \Lambda^{k - i} V
  \]
  in exterior algebra.
\end{lemma}

\subsection{Semiorthogonality of the non-rectangular part}
\label{ssec:lefschetz tail semiorthogonality}

In this subsection we prove the third claim of Theorem~\ref{thm:lefschetz grouping}. We will again rely on the results from Section~\ref{sec:symmetric algebra exterior square}. We need to prove the following vanishing.

\begin{lemma}
  \label{lem:lefschetz tail}
  Let $0 \leq i \leq n-1$ be an index, and let $k \in [ \binom{2n-1}{2}-1, \binom{2n-1}{2}+n-2]$. Then
  \[
    \RGamma(X_1, L_i(-kH)) = 0.
  \]
\end{lemma}
\begin{proof}
  We can compute (hyper)cohomology on $X_1$ by first pushing the object down along the morphism $\pi\colon X_1 \to \P(V)$:
  \begin{equation}
    \label{eq:lefschetz tail vanishing}
    \begin{split}
      \RGamma(X_1, & L_i(-kH)) \caniso \RGamma(\P(V), \pi_*(L_i(-kH))) \caniso \\
      & \caniso \RGamma(\P(V), [ \Lambda^{2n-i}Q \otimes \pi_*\O(-(k+n-i)H) \to \Lambda^iQ \otimes \pi_*\O(-kH) ]).
    \end{split}
  \end{equation}
  First, we will prove the vanishing of cohomology~\eqref{eq:lefschetz tail vanishing} in the special case $k = \binom{2n-1}{2} - 1$, the smallest allowed value of~$k$. It is distinguished by the fact that by Lemma~\ref{lem:pushforwards fiberwise} the pushforward~$\pi_*\O(-kH)$ is a zero object in this case, while it is nonzero for any~$k > \binom{2n-1}{2}-1$. Note that since~$i \leq n-1$ and~$k \geq \binom{2n-1}{2}-1$, the first twist~$-(k+n-i)$ is always a negative number of magnitude~$\geq \binom{2n-1}{2}$ and thus the pushforward~$\pi_*\O(-(k+n-i)H)$ is always described by the formula from the third part of Lemma~\ref{lem:pushforwards fiberwise}. Due to the vanishing of the pushforward~$\pi_*\O(-kH)$ in this special case it remains to prove the vanishing of
  \[
    \begin{aligned}
      \RGamma(\P(V), & \,\, \Lambda^{2n-i} Q \otimes \pi_*\O(-(\tbinom{2n-1}{2} + n - i - 1)H)) \caniso \\
      & \caniso \RGamma(\P(V), \O(-(2n-2)) \otimes \Lambda^{2n-i} Q \otimes \Sym^{n - i - 1}(\Lambda^2 Q^\dual)) \caniso \\
      & \caniso \RGamma(\P(V), \O(-(2n-3)) \otimes \Lambda^{i-1} Q^\dual \otimes \Sym^{n-i-1}(\Lambda^2 Q^\dual)).
    \end{aligned}
  \]
  Here we used Lemma~\ref{lem:pushforwards fiberwise} for the first isomorphism and the identity
  \[
    \Lambda^{2n-i}Q \caniso \Lambda^{i-1} Q^\dual \otimes \det(Q) \caniso \Lambda^{i-1} Q^\dual \otimes \O_{\P(V)}(1)
  \]
  for the second one.
  We can assume that $i > 0$ since for~$i = 0$ the bundle~$\Lambda^{2n-i} Q$ is zero and there is nothing to prove. Any Young diagram appearing in the decomposition of the tensor product~$\Lambda^{i-1} Q^\dual \otimes \Sym^{n-i-1}(\Lambda^2 Q^\dual)$ has~$2n - 3 - i$ boxes, and for~$i > 0$ this is strictly smaller than~$2n-3$, and then the vanishing of cohomology holds by Corollary~\ref{cor:bwb acyclicity}.

  Since the vanishing has been proved for $k = \binom{2n-1}{2}-1$ we can assume that $k = \binom{2n-1}{2} + k^\prime$, where $k^\prime \in [0; n-2]$. Applying Lemma~\ref{lem:pushforwards fiberwise} to~\eqref{eq:lefschetz tail vanishing} we see that we need to prove the vanishing of the cohomology for the object
  \[
    \begin{split}
      [ \Lambda^{2n-i} & Q \otimes \O(-(2n-2)) \otimes \Sym^{n-i+k^\prime}(\Lambda^2 Q^\dual) \to \Lambda^i Q \otimes \O(-(2n-2)) \otimes \Sym^{k^\prime}(\Lambda^2 Q^\dual) ] \caniso \\
      & \caniso \O(-(2n-3)) \otimes [ \Lambda^{i-1} Q^\dual \otimes \Sym^{n-i+k^\prime}(\Lambda^2 Q^\dual) \to \Lambda^{2n-2 - (i-1)} Q^\dual \otimes \Sym^{k^\prime}(\Lambda^2 Q^\dual) ].
    \end{split}
  \]
  Note that the morphism
  \begin{equation}
    \label{eq:lefschetz tail object}
    \Lambda^{i-1} Q^\dual \otimes \Sym^{n-i+k^\prime}(\Lambda^2 Q^\dual) \xrightarrow{\psi} \Lambda^{2n-2 - (i-1)} Q^\dual \otimes \Sym^{k^\prime}(\Lambda^2 Q^\dual),
  \end{equation}
  which we will call $\psi$, 
  is by Lemma~\ref{lem:lis and tautological morphisms} equal to the dual to the morphism $\varphi_{n-1, i-1, k^\prime}(Q)$ from Proposition~\ref{prop:phi maps}. As such, by construction the cone of the morphism~\eqref{eq:lefschetz tail object} is isomorphic to the direct sum~$\ker(\psi)[1] \oplus \coker(\psi)$, and both the kernel and cokernel are isomorphic to direct sums of some Schur functors applied to the vector bundle~$Q^\dual$. Thus what we need to prove is that for any Young diagram $\lambda$ appearing in the decomposition of either kernel or cokernel of the map $\psi$ with nonzero multiplicity the cohomology
  \begin{equation}
    \label{eq:cohomology of twist should vanish}
    \RGamma(\P(V), \O(-(2n-3)) \otimes \Sigma^\lambda Q^\dual)
  \end{equation}
  vanishes.

  Note that if the number of rows in the diagram $\lambda$ is $2n-4$ or less, the vanishing of~\eqref{eq:cohomology of twist should vanish} is automatic by Corollary~\ref{cor:bwb acyclicity}. More generally, from the Borel--Weil--Bott theorem~\ref{thm:bwb} we conclude that the cohomology~\eqref{eq:cohomology of twist should vanish} vanishes if and only if the diagram $\lambda$ satisfies at least one of the following conditions:
  \begin{itemize}
  \item The $(2n-3)$'rd row of $\lambda$ is empty;
  \item The $(2n-2)$'nd row of $\lambda$ has exactly one box;
  \item The $(2n-1)$'st row of $\lambda$ has exactly two boxes.
  \end{itemize}

  The key observation is that since~$\psi$ is the dual map to~$\varphi_{n-1, i-1, k^\prime}$, by Corollary~\ref{cor:cone of tautological morphism} in the decompositions of the kernel and the cokernel of~$\psi$ there are no diagrams~$\lambda$ with~$2n-2$ or~$2n-3$ rows. As such, only a diagram $\lambda$ with exactly $2n-1$ rows may violate the vanishing~\eqref{eq:cohomology of twist should vanish}. By looking at the map~\eqref{eq:lefschetz tail object} and using the bound $k^\prime \leq n-2$ we see that the number of boxes in any diagram $\lambda$ appearing with nonzero multiplicity is
  \[
    i - 1 + 2(n-i+k^\prime) = 2n - i  - 1 + 2 k^\prime \leq 4n - 5 - i.
  \]
  Since we assume that $\lambda$ has $2n-1$ rows, its' $(2n-2)$'nd row must be nonempty; however, now it is clear that the $(2n-2)$'nd row of $\lambda$ cannot have two or more boxes since any such diagram has at least $2 \cdot (2n-2) = 4n - 4 > 4n-5 -i$ boxes. Thus, for any $\lambda$ appearing in~\eqref{eq:lefschetz tail object} with nonzero multiplicity and having $2n-1$ rows we know that its $(2n-2)$'nd row has exactly one box, and then the vanishing~\eqref{eq:cohomology of twist should vanish} holds by Borel--Weil--Bott theorem, as noted above.
\end{proof}

This finishes the proof of Theorem~\ref{thm:lefschetz grouping}: we have proved each of three claims of the theorem separately, heavily relying on the facts established in Section~\ref{sec:symmetric algebra exterior square}, mostly on Proposition~\ref{prop:phi maps} and Corollary~\ref{cor:cone of tautological morphism}.

\section{Conjectures}
\label{sec:conjectures}

In Theorem~\ref{thm:lefschetz grouping} we have explicitly constructed some admissible subcategory $\mA \subset \Dbcoh(X_1)$ and found a Lefschetz exceptional collection in it. As discussed in the introduction, this subcategory is a candidate for being a geometric realization of the homologically projectively dual category to the standard Lefschetz collection on $\Gr(2, V)$. However, there are quite a lot of things left to prove in order to complete this picture. In this section we list some desired properties of the category $\mA$ that we conjecture to be true.

\subsection{Linearity and homological projective duality}
\label{ssec:conj linearity}
The category $\mA$ should be linear with respect to the projection morphism~$p\colon X_1 \to \P(\Lambda^2 V^\dual)$. This means that for any $i \in [0, n-1]$ and any integer $k \in \Z$ we expect that the twist $L_i(kH) \in \Dbcoh(X_1)$ lies in the subcategory $\mA$, not only for the values of~$k$ mentioned in Theorem~\ref{thm:lefschetz grouping}. Linearity is, of course, required if we want to say anything related to the HPD interpretation of the category $\mA$.
We conjecture that the category $\mA$ is, in fact, the homologically projectively dual category to the Grassmannian $\Gr(2, V)$:

\begin{conjecture}
  \label{conj:main conjecture}
  The subcategory~$\mA \subset \Dbcoh(X_1)$ is linear with respect to the projection morphism~$p\colon X_1 \to \P(\Lambda^2 V^\dual)$, and there exists a~$\Dbcoh(\P(\Lambda^2 V^\dual))$-linear equivalence between~$\mA$ and the homologically projectively dual category to~$\Gr(2, V)$.
\end{conjecture}

The exceptional collection in~$\mA$ from Theorem~\ref{thm:lefschetz grouping} has expected number of semiorthogonal blocks of expected lengths, so this seems possible. Since homologically projectively dual category has no description by a universal property (to the best of my knowledge), the only way to prove a statement like that would be construct a~$\P(\Lambda^2 V^\dual)$-linear equivalence between~$\mA$ and a certain~$\P(\Lambda^2 V^\dual)$-linear subcategory in the universal hyperplane section variety $\mathcal{H}_{\Gr(2, V)}$ of the Pl\"ucker embedding~$\Gr(2, V) \subset \P(\Lambda^2 V)$ defined explicitly in the main theorem of homological projective duality (see \cite[Def.~6.1]{hpd} or \cite[Def.~7.1]{noncommutativehpd}).

Note that the variety $\mathcal{H}_{\Gr(2, V)}$ is the space of pairs, a~$2$-form~$\omega$ on~$V$ and a~$2$-dimensional subspace of~$V$ which is isotropic with respect to~$\omega$. There is a natural correspondence between this space and~$X_1$: given a point on~$X_1$, i.e., a degenerate~$2$-form~$\omega$ on~$V$ and a one-dimensional subspace~$K_1 \subset V$ in its kernel, any~$2$-dimensional subspace of~$V$ containing~$K_1$ will be isotropic with respect to~$\omega$. So there are certainly some interesting~$\P(\Lambda^2 V^\dual)$-linear functors between~$\Dbcoh(X_1)$ and~$\Dbcoh(\mathcal{H}_{\Gr(2, V)})$. What happens with the subcategory $\mA \subset \Dbcoh(X_1)$ under the action of such functors is far from obvious. Figuring out which functor is the ``correct one'' and proving enough properties to establish $\mA$ as the homologically projectively dual category to $\Gr(2, V)$ seems to be quite difficult.

\subsection{Semiorthogonal decomposition of $X_1$}
\label{ssec:x1 decomposition}
There is a not quite formal argument that whatever the homologically projectively dual category to $\Gr(2, V)$ is, the category $\Dbcoh(X_1)$ should be glued from two copies of that category. This argument was communicated to the author by Alexander Kuznetsov and served as the original motivation for this work. The argument goes as follows.

Consider the partial flag variety $\Fl(1, 2; V)$ and the map $f\colon \Fl(1, 2; V) \to \P(\Lambda^2 V)$ defined as the composition
\[
  \Fl(1, 2; V) \to \Gr(2, V) \monoarrow \P(\Lambda^2 V).
\]
The map~$f$ lets us think of~$\Dbcoh(\Fl(1, 2; V))$ as a category linear over~$\P(\Lambda^2 V)$. The theory of homological projective duality works with what can be called \emph{Lefschetz categories}, i.e., categories equipped with a Lefschetz semiorthogonal decomposition with respect to the action of the derived category of some projective space. The category~$\Dbcoh(\Fl(1, 2; V))$ is not naturally a Lefschetz category in that sense; instead, since~$\Fl(1, 2; V)$ can be identified with the projectivization~$\P_{\Gr(2, V)}(U)$ of the tautological rank-$2$ subbundle~$U$ on~$\Gr(2, V)$, the category~$\Dbcoh(\Fl(1, 2; V))$ has an~$f$-linear semiorthogonal decomposition into two copies of the Lefschetz category~$\Dbcoh(\Gr(2, V))$. If we had some version of homological projective duality that applied not only to Lefschetz categories, but also to categories linearly glued from several Lefschetz categories, then presumably the HPD category to~$\Fl(1, 2; V)$ would be glued from two copies of HPD categories for~$\Gr(2, V)$. To the best of author's knowledge, there are no known theories of ``relative homological projective duality'' flexible enough to make this a rigorous statement. Now, geometric arguments similar to the one showing that the HPD category for~$\Gr(2, V)$ should be (a noncommutative resolution of) the Pfaffian variety in~$\P(\Lambda^2 V^\dual)$ suggest that the HPD category to~$\Fl(1, 2; V)$ in this framework should be the derived category of~$X_1$. And then, if we believe all those ``should''s, there would be a semiorthogonal decomposition for~$\Dbcoh(X_1)$ into two copies of the HPD category for~$\Gr(2, V)$.

We conjecture that there is a semiorthogonal decomposition
\[
  \Dbcoh(X_1) = \langle \mA, \mA(h) \rangle,
\]
where, as usual,~$\O(h)$ is the pullback of~$\O_{\P(V)}(1)$ along the map~$\pi\colon X_1 \to \P(V)$. On the level of~$K_0$ this appears to be true. One can show by an inductive argument that if~$\mA$ is, indeed, a~$\P(\Lambda^2 V^\dual)$-linear category (see Subsection~\ref{ssec:conj linearity}), then the pair of categories~$\mA$ and~$\mA(h)$ together generate~$\Dbcoh(X_1)$. The semiorthogonality between them is not obvious. There are some easy computations one can make involving the object $L_0 := \O_{X_1}$:

\begin{lemma}
  For any $k \in \Z$ we have semiorthogonality $\O(h + kH) \in {}^\perp\mA$ and $\O(kH) \in \mA(h)^\perp$.
\end{lemma}
\begin{proof}
  By definition of the category $\mA$ the first semiorthogonality is equivalent to the fact that for any $i \in [0; n-1]$ and any $k$ we have the vanishing
  \begin{equation}
    \label{eq:p1-fibration first}
    \RHom_{X_1}(\O(h), L_i(kH)) \caniso \RGamma(X_1, L_i(kH - h)) = 0.
  \end{equation}
  Note that since $\O(H)$ is the pullback of the ample line bundle from $\P(\Lambda^2 V^\dual)$ any $\binom{2n}{2}$ consequitive powers of $\O(H)$ generate the subcategory spanned by all $\O(kH)$ for $k \in \Z$. Thus, in order to prove the vanishing~\eqref{eq:p1-fibration first} for all values of $k$, it suffices to prove it for sufficiently large $k$. For~$k \geq n-i$ by the first part of Lemma~\ref{lem:lis and tautological morphisms} we need to show the vanishing of
  \begin{equation}
    \label{eq:computing p1-fibration first}
    \RGamma(\P(V), \O(-1) \otimes
    [
    \Lambda^{2n-i} Q \otimes \Sym^{k-n + i}(\Lambda^2 Q)
    \xrightarrow{\varphi_{n, i, k - n + i}(Q)}
    \Lambda^i Q \otimes \Sym^k(\Lambda^2 Q)
    ]
    ).
  \end{equation}
  We know that the cone of $\varphi_{n, i, k-n+i}(Q)$ splits into a direct sum of Schur functors of $Q$ where all Young diagrams appearing with nonzero multiplicity have either strictly more than $2n$ rows or strictly less than $2n-1$ rows (Corollary~\ref{cor:cone of tautological morphism}). Since the rank of $Q$ is $2n-1$, we see that only diagrams with at most $2n-2$ rows appear in the cone. Then by Corollary~\ref{cor:bwb acyclicity} the twist of the cone by $\O(-h)$ is acyclic for all $k \gg 0$ and the proof is finished.

  Similarly, for the second semiorthogonality it is enough to prove that
  \[
    \RHom_{X_1}(L_i(h), \O(kH)) = 0
  \]
  for all sufficiently negative values of $k$. By Serre duality on $X_1$ (Lemma~\ref{lem:serre duality on x1}) this graded space is dual to
  \[
    \RGamma(X_1, \O(-h) \otimes L_i(-(k + \tbinom{2n-1}{2})H)).
  \]
  Note that for $k \ll 0$ this becomes the same space as in~\eqref{eq:p1-fibration first} above, and we already proved its vanishing.
\end{proof}

\subsection{Same subcategory appears in $X_2$}
\label{ssec:comparison with x2}
Consider the Grassmannian $\Gr(2, V)$ and let us denote by $\tilde{Q}$ the tautological quotient bundle on it. Recall that the variety $X_2$ was defined as the projectivization $\P_{\Gr(2, V)}(\Lambda^2 \tilde{Q}^\dual)$, and let us denote by $\tilde{\pi}\colon X_2 \to \Gr(2, V)$ the projection morphism, with $\tilde{H}$ being the relative ample divisor. We can easily define analogues of the objects $L_i$ on $X_2$:
\[
  \widetilde{L_i} := [ \tilde{\pi}^* \Lambda^{2n-i} \tilde{Q} \otimes \O_{X_2}(-(n-i)\tilde{H}) \to \tilde{\pi}^* \Lambda^i \tilde{Q} ],
\]
where the map is, as for $L_i$'s, given by the multiplication by the~$(n-i)$'th power of the~$2$-form on~$\tilde{Q}$. It is easy to see that the proof of Lemma~\ref{lem:li are exceptional} works on $X_2$ as well, i.e., that each~$\widetilde{L_i}$ is an exceptional coherent sheaf on~$X_2$. We conjecture that the analogue of Theorem~\ref{thm:lefschetz grouping} holds on~$X_2$. Namely, that the collection
\[
  \langle \widetilde{L_0}, \ldots, \widetilde{L_{n-1}} \rangle
\]
is exceptional, and if we denote by $\widetilde{\mA_0}$ the spanned subcategory, then
\[
  \langle \widetilde{\mA_0}, \widetilde{\mA_0}(\tilde{H}), \ldots, \widetilde{\mA_0}((\tbinom{2n-1}{2} - 2)\tilde{H}), \O_{X_2}((\tbinom{2n-1}{2}-1)\tilde{H}), \ldots, \O_{X_2}((\tbinom{2n-1}{2} + n - 2)\tilde{H}) \rangle
\]
is a semiorthogonal collection. Moreover, we expect that the generated subcategory should be equivalent to the category $\mA \subset \Dbcoh(X_1)$.

Numerically, on the level of $K_0$ (or even the $\GL(V)$-equivariant $K_0$), this seems to be true, at least for small values of $n$. One complication that appears on $X_2$ as compared to Theorem~\ref{thm:lefschetz grouping} on $X_1$ boils down to the fact that since the rank of $\Lambda^2 \tilde{Q}^\dual$ is $\binom{2n-2}{2} < \binom{2n-1}{2}$, there are fewer integers $k$ such that $\tilde{\pi}_*\O(k \tilde{H})$ is a zero object (see Lemma~\ref{lem:pushforwards fiberwise}). For example, we have seen in Lemma~\ref{lem:ilessj easy part} that for indices~$0 \leq i \leq j \leq n-1$ and an integer~$k \in \Z$ in the rectangular range the object~$\Lambda^i Q$ on~$X_1$ is left-orthogonal not only to the object~$L_j(-kH)$, but, in fact, to both vector bundles appearing in the definition of~$L_j(-kH)$. An analogue of this strengthening of semiorthogonality on~$X_1$ is no longer true on~$X_2$ for the values of~$k$ between~$\binom{2n-2}{2}$ and~$\binom{2n-1}{2} - 2$ due to the non-vanishing of some pushforwards.

Nevertheless, the semiorthogonality
\[
  \RHom_{X_2}(\tilde{\pi}^* \Lambda^i \tilde{Q}, \widetilde{L_j}(-k\tilde{H})) = 0
\]
analogous to the conclusion of Lemma~\ref{lem:ilessj easy part} can still be proved by using the same tools as in Section~\ref{sec:lefschetz rectangular by computation}, but the computations get messy: there are many different cases to consider, and one needs to make numerous arguments of the type ``if~$\mu$ is a Young diagram with~$2n-2$ rows and at most~$6n - 10 - j$ boxes such that~$\mu_{2n-2} = 1$ and~$\mu_{2n-3} = 2$, then by box-counting reasons the third column of~$\mu$ is shorter than the second by at least~$j + 2$, so since~$i < j + 2$ by assumption, by the Pieri's formula any diagram~$\lambda$ appearing with nonzero multiplicity in the tensor product~$\Lambda^i Q^\dual \otimes \Sigma^\mu Q^\dual$ has~$\lambda_{2n-3} = \mu_{2n-3} = 2$'' (compare the end of the proof of Lemma~\ref{lem:lefschetz tail}), and so on. It seems plausible that a full analogue of Theorem~\ref{thm:lefschetz grouping} on $X_2$ can be proved using only the insights presented in Section~\ref{sec:symmetric algebra exterior square} and a lot of routine combinatorics, but there is probably a more conceptual way. Or, at least, a nicer way to package all the combinatorial cases.

Besides, conjecturally there should exist a fully faithful functor~$\mA \monoarrow \Dbcoh(X_2)$ that sends objects~$L_i$ to~$\widetilde{L_i}$, and the construction of such a functor requires more than just direct combinatorial computations on~$X_2$.

\printbibliography

\end{document}